\documentclass[final]{siamltex}

\usepackage{subdepth}
\usepackage{overpic}
\usepackage{amsmath,amstext,amsbsy,amssymb,subfig}
\usepackage{booktabs,bm}
\usepackage{mathrsfs}
\usepackage{tikz}
\usetikzlibrary{positioning}
\usetikzlibrary{shapes,arrows}
\usepackage[fleqn,tbtags]{mathtools}
\usepackage{relsize}
\newtheorem{remark}{Remark}[section]

\title{A fast analysis-based discrete Hankel transform using asymptotic expansions}
\author{Alex Townsend\thanks{Department of Mathematics, Massachusetts Institute of Technology, 77 Massachusetts Avenue
Cambridge, MA 02139-4307. (ajt@mit.edu)}}

\begin{document}
\maketitle

\begin{abstract}
A fast and numerically stable algorithm is described for computing the discrete Hankel transform 
of order $0$ as well as evaluating Schl\"{o}milch and Fourier--Bessel expansions in
$\mathcal{O}(N(\log N)^2/\log\!\log N)$ operations. The algorithm is based on an 
asymptotic expansion for 
Bessel functions of large arguments, the fast Fourier transform, and the Neumann addition 
formula. All the algorithmic parameters are selected from 
error bounds to achieve a near-optimal computational cost for any accuracy goal. 
Numerical results demonstrate the efficiency of the resulting algorithm. 
\end{abstract}



\section{Introduction} 
The Hankel transform of order $\nu\geq0$ of a function $f\!:\![0,1]\rightarrow\mathbb{C}$ is defined as~\cite[Chap.~9]{Poularikas_10_01}  
\begin{equation}
F_{\nu}(\omega) = \int_0^1 f(r) J_{\nu}( r\omega ) r \, dr, \qquad \omega > 0,
\label{eq:HankelTransform} 
\end{equation} 
where $J_{\nu}$ is the Bessel function of the first kind with parameter $\nu$. 
When $\nu=0$,~\eqref{eq:HankelTransform} is essentially the 
two-dimensional Fourier transform of a radially symmetric function with support 
on the unit disc and appears in the Fourier--Hankel--Abel cycle~\cite[Sec.~9.3]{Poularikas_10_01}. 
Furthermore, the Hankel transform of integer order $\nu\geq 1$ is equivalent to 
the Fourier transform of functions of the form $e^{{\rm i}\nu\theta}f(r)$, where 
$\theta\in[0,2\pi)$. Throughout
this paper we take $\nu$ to be an integer. 

Various semi-discrete and discrete variants of~\eqref{eq:HankelTransform}  
are employed in optics~\cite{Guizar_04_01}, electromagnetics~\cite{Nachamkin_80_01},
medical imaging~\cite{Higgins_88_01}, and the numerical solution of partial differential 
equations~\cite{Bisseling_85_01}. At some stage these require the computation of the following sums:
\begin{equation}
 f_k = \sum_{n=1}^N c_{n} J_{\nu}( r_k \omega_n  ), \qquad 1\leq k\leq N,
\label{eq:discreteHankelTransforms}
\end{equation} 
where $\nu$ is an integer, $0<\omega_1<\cdots <\omega_n<\infty$, and 
$0 \leq r_1<\cdots < r_N\leq 1$. Special cases of~\eqref{eq:discreteHankelTransforms}
that are considered in this paper are: (1) The evaluation of Schl\"{o}milch expansions at $r_k$~\cite{Schlomilch_1846_01} 
($\omega_n = n\pi$ or $\omega_n = (n+\nu)\pi$); (2) 
The evaluation of Fourier--Bessel expansions of order $0$ 
at $r_k$ ($\nu=0$ and $\omega_n=j_{0,n}$, where $j_{0,n}$ is the $n$th positive root of $J_{0}$); and
(3) The discrete Hankel transform of order $0$~\cite{Johnson_87_01} ($\nu=0$, $\omega_n=j_{0,n}$, and $r_k = j_{0,k}/j_{0,N+1}$).
For algorithms that rely on the fast Fourier transform (FFT), such as the algorithm described in 
this paper, tasks (1), (2), and (3) are incrementally more difficult. 
We consider them in turn in sections~\ref{sec:Schlomilch},~\ref{sec:FourierBessel}, 
and~\ref{sec:DiscreteHankelTransform}, respectively. 

Our algorithm is based on carefully replacing $J_\nu(z)$ by an asymptotic expansion for
large arguments that, up to a certain error, expresses $J_\nu(z)$ as a weighted sum of 
trigonometric functions. This allows a significant proportion of the computation involved 
in~\eqref{eq:discreteHankelTransforms} to be achieved by the discrete cosine transform of type I (DCT) 
and the discrete sine transform of type I (DST). These 
are $\mathcal{O}(N\log N)$ FFT-based algorithms for computing matrix-vector products with 
the matrices $C^{I}_N$ and $S^I_N$, where
\begin{equation}
 (C^I_N)_{kn} = \cos\left(\frac{(k-1)(n-1)\pi}{N-1}\right), \qquad (S^I_N)_{kn} = \sin\left(\frac{kn\pi}{N+1}\right), \qquad 1\leq k,n\leq N.
\label{eq:DCTs}
\end{equation} 

Asymptotic expansions of special functions are notoriously tricky to use in practice 
because it is often difficult to determine the precise regime for which they are accurate approximations. 
However, for transforms involving orthogonal polynomials, such as the Chebyshev--Legendre 
transform~\cite{Hale_14_01,Mori_99_01}, successful fast algorithms have been derived based 
on asymptotic expansions. At the point of evaluation, these algorithms are competitive 
with multipole-like approaches~\cite{Alpert_91_01}, but have the advantage of no precomputation. 
This allows for more adaptive algorithms such as a convolution algorithm for 
Chebyshev expansions~\cite{Hale_14_02}. 

In this paper we will use error bounds to determine the precise regime 
for which an asymptotic expansion of Bessel functions is an
accurate approximation and from there derive all the algorithmic 
parameters to achieve a near-optimal computational cost for any
accuracy goal (see Table~\ref{tab:algorithmicParameters}). For each task 
(1)-(3) there are just two user-defined inputs: A vector of coefficients $c_1,\ldots,c_N$ in~\eqref{eq:discreteHankelTransforms} 
and a working accuracy $\epsilon>0$. The values of $f_k$ 
in~\eqref{eq:discreteHankelTransforms} are then calculated to an accuracy of 
$\mathcal{O}(\epsilon \sum_{n=1}^N|c_n|)$. 
The resulting algorithm has a complexity of $\mathcal{O}(N(\log N)^2\log(1/\epsilon)^p/\log\!\log N)$, 
where $p=1$ for task (1), $p=2$ for task (2), and $p=3$ for task (3). 
For the majority of this paper we state algorithmic complexities without the 
dependency on $\epsilon$.

Previous algorithms based on asymptotic expansions 
for computing~\eqref{eq:discreteHankelTransforms}~\cite{Candel_81_01,Sharafeddin_92_01}  have either a 
poor accuracy or a dubious numerical stability~\cite{Cree_93_01}. 
Here, we find that once an asymptotic expansion of Bessel functions is carefully 
employed the resulting algorithm is numerically stable, can be adapted to 
any accuracy goal, and requires no precomputation. Moreover, we show how 
the equally-spaced restriction, which is inherited from the reliance on DCTs and DSTs, can be alleviated. 

We use the following notation. A column vector with 
entries $v_1,\ldots,v_N$ is denoted by $\underline{v}$, a row vector is denoted 
by $\underline{v}^\intercal$, a diagonal matrix 
with diagonal entries $v_1,\ldots,v_N$ is written as $D_{\underline{v}}$, and 
the $N\times N$ matrix with $(k,n)$ entry $J_\nu(a_{kn})$ is 
denoted by $\mathbf{J}_\nu(A)$ where $A=(a_{kn})_{1\leq k,n\leq N}$.

The paper is organized as follows. In Section~\ref{sec:existing} we briefly discuss 
existing methods for computing~\eqref{eq:discreteHankelTransforms} and in 
Section~\ref{sec:BesselFunctionProperties} we investigate the approximation power of 
three expansions for Bessel functions. In Section~\ref{sec:Schlomilch} we first describe an 
$\mathcal{O}(N^{3/2})$ algorithm for evaluating Schl\"{o}milch expansions before
deriving a faster $\mathcal{O}(N(\log N)^2/\log\!\log N)$ algorithm. 
In Sections~\ref{sec:FourierBessel} and~\ref{sec:DiscreteHankelTransform} 
we use the Neumann addition formula to extend the algorithm to the 
evaluation of Fourier--Bessel expansions and 
the computation of the discrete Hankel transform. 

\section{Existing methods}\label{sec:existing}
We now give a brief description of some existing 
methods for evaluating~\eqref{eq:discreteHankelTransforms} that roughly fall into four categories: 
(1) Direct summation; 
(2) Evaluation via an integral representation; 
(3) Evaluation using asymptotic expansions, and; (4) Butterfly schemes.  For
other approaches see~\cite{Cree_93_01}.

\subsection{Direct summation}
One of the simplest ideas is to evaluate $J_{\nu}(r_k\omega_n)$ for $1\leq k,n\leq N$ and 
then naively compute the sums in~\eqref{eq:discreteHankelTransforms}. 
By using the Taylor series expansion in~\eqref{eq:TaylorSeriesInfinite} 
for small $z$, the asymptotic 
expansion in~\eqref{eq:HankelExpansion} for large $z$, and Miller's algorithm~\cite[Sec.~3.6(vi)]{NISTHandbook} (backward 
recursion with a three-term recurrence) otherwise, 
each evaluation of $J_{\nu}(z)$ costs $\mathcal{O}(1)$ operations~\cite{Amos_86_01}.
Therefore, this simple algorithm requires a total of $\mathcal{O}(N^2)$ operations and is 
very efficient for small $N$ ($N\leq 256$). However, the $\mathcal{O}(N^2)$ 
cost is prohibitive for large $N$.

Note that in general there is no evaluation scheme for $J_{\nu}(z)$ that is exact 
in infinite precision arithmetic. At some level a working accuracy of $\epsilon>0$ 
must be introduced. Thus, even in the direct summation approach, the
relaxation of the working accuracy that is necessary for deriving a faster algorithm, 
has already been made.

The sums in~\eqref{eq:discreteHankelTransforms} are equivalent to the matrix-vector product 
$\mathbf{J}_\nu(\, \underline{r}\,\underline{\omega}^\intercal)\underline{c}$. One reason that an 
$\mathcal{O}(N(\log N)^2/\log\!\log N)$ algorithm is possible is because the argument, $\underline{r}\,\underline{\omega}^\intercal$, 
is a rank $1$ matrix. A fact that we will exploit several times. 

\subsection{Evaluation via an integral representation}
When $\nu$ is an integer, $J_\nu(z)$ can
be replaced by its integral representation~\cite[(10.9.2)]{NISTHandbook}, 
\begin{equation}
 J_\nu(z) = \frac{{\rm i}^{-\nu}}{\pi}\int_0^{\pi} e^{{\rm i}z\cos(\theta)}\cos( \nu\theta) d\theta = \frac{{\rm i}^{-\nu}}{2\pi}\int_0^{2\pi} e^{{\rm i}z\cos(\theta)}\cos( \nu\theta) d\theta, 
\label{eq:BesselIntegralRepresentation}
\end{equation}
where ${\rm i}=\sqrt{-1}$ is the imaginary unit and the last 
equality holds because $e^{{\rm i}z\cos(\theta)}\cos( \nu\theta)$ is even about $t=\pi$. 
The integrand in~\eqref{eq:BesselIntegralRepresentation} is $2\pi$-periodic and the integral 
can be approximated by a $K$-point trapezium rule.  
Substituting this approximation into~\eqref{eq:discreteHankelTransforms} gives a double 
sum for $\underline{f}$ that can be evaluated with one FFT followed
by an interpolation step in $\mathcal{O}(KN + N\log N)$ 
operations~\cite{Gopalan_83_01}.  The major issue with this approach 
is that the integrand in~\eqref{eq:BesselIntegralRepresentation}
is highly oscillatory for large $z$ and often $K\approx N$, resulting 
in $\mathcal{O}(N^2)$ operations.

\subsection{Evaluation using asymptotic expansions}
One can replace $J_\nu(z)$ in~\eqref{eq:discreteHankelTransforms}
by an asymptotic expansion that involves trigonometric 
functions. A fast, but perhaps numerically unstable  (see Figure~\ref{fig:partition}), 
algorithm then follows by exploiting DCTs and DSTs.  
Existing work has approximated $J_{\nu}(z)$ using asymptotic expansions 
with rational coefficients~\cite{Piessens_82_01}, the approximation 
$\sqrt{2/(\pi z)}\cos(z-(2\nu+1)\pi/4)$~\cite{Candel_81_01}, and for 
half-integer $\nu$ the asymptotic expansion 
in~\eqref{eq:HankelExpansion}~\cite{Sharafeddin_92_01}. However, none 
of these approaches have rigorously determined the regime in which 
the employed asymptotic expansion is an accurate approximation and instead 
involve dubiously chosen algorithmic parameters. In this paper we will use 
a well-established asymptotic expansion and rigorously determine the regime in which 
it is an accurate approximation. 

\subsection{A fast multipole method approach} 
In 1995, Kapur and Rokhlin described an algorithm for computing the integral 
in~\eqref{eq:HankelTransform} when $\nu = 0$~\cite{Kapur_95_01}. 
The main idea is to write~\eqref{eq:HankelTransform} as a product of 
two integrals as follows~\cite[(6)]{Kapur_95_01}:
\[
 \int_0^1 f(r) J_0(r\omega)r\,dr = \frac{1}{\pi}\int_{-\omega}^\omega \frac{1}{\sqrt{\omega^2-u^2}} \int_0^1 rf(r)\cos(ru) \, dr\,du.
\]
The inner integral can be discretized by a (corrected) trapnezium rule and 
computed via a discrete cosine transform, while the outer integral involves 
a singular kernel and can be computed with the fast multipole method. 
As described in~\cite{Kapur_95_01} the algorithm can be used to evaluate Schl\"{o}milch expansions
in $\mathcal{O}(N\log N)$ operations (see Section~\ref{sec:Schlomilch}). With modern 
advances in fast algorithms for nonuniform transforms, it could now be adapted
to the computation of the discrete Hankel transform though we are not aware of such 
work. 

\subsection{Butterfly schemes}
Recently, a new algorithm was developed for a class of fast transforms
involving special functions~\cite{ONeil_10_01}. The algorithm compresses the 
ranks of submatrices of a change of basis matrix by constructing 
skeleton decompositions
in a hierarchical fashion. This hierarchical decomposition of the matrix 
means that it can then be applied to a vector in $\mathcal{O}(N\log N)$ operations. 
The compression step of the algorithm is the dominating cost and 
is observed to require $\mathcal{O}(N^2)$ operations~\cite[Table.~6]{ONeil_10_01}. 
Fortunately, in many applications the compression step can be regarded as a 
precomputational cost because it is independent of the values of
$c_1,\ldots,c_N$ in~\eqref{eq:discreteHankelTransforms}, though it does 
depend on $N$.  

The methodology behind butterfly schemes is more general than what we present in this 
paper as it does not require knowledge of asymptotics. 
Here, we will derive a fast transform 
that has no precomputational cost and is better suited to applications where $N$
is not known in advance.

Some of the ideas in the 
spherical harmonic transform literature~\cite{Rokhlin_06_01,Tygert_08_01}
may be useful for reducing the precomputational cost in~\cite{ONeil_10_01}, though 
we are not aware of work in this direction. 

\section{Three expansions of Bessel functions}\label{sec:BesselFunctionProperties}
In this section we investigate three known expansions of Bessel functions~\cite[Chap.~10]{NISTHandbook}: (1) 
An explicit asymptotic expansion for large arguments; (2) A convergent Taylor series 
expansion for small arguments; and (3) The Neumann addition formula for perturbed arguments. 
Mathematically, these three expansions are infinite series, but for practical use they must 
be truncated. Here, we will focus on deriving criteria to ensure that 
the errors introduced by such truncations are negligible. Related techniques 
were successfully used to derive a fast algorithm for the discrete Legendre 
transform~\cite{Hale_15_01}. 

\subsection{An asymptotic expansion of Bessel functions for large arguments}\label{sec:asymptoticExpansion}
The graph of $J_\nu(z)$ on $[0,\infty)$ is often said to look like an oscillating 
trigonometric function that decays like $1/\sqrt{z}$. This observation can 
be made precise by an asymptotic expansion of $J_\nu(z)$ for large $z$. 

Let $\nu$ and $M\geq 1$ be integers and $z>0$. The first $2M$ terms of the Hankel
asymptotic expansion of $J_\nu(z)$ is given by~\cite[(10.17.3)]{NISTHandbook} (also see~\cite{Hankel_1869_01}):
\begin{equation}
J_\nu(z) = \left(\frac{2}{\pi z} \right)^{\!\!\frac{1}{2}}\left(\cos(\mu)\! \!\sum_{m = 0}^{M-1} (-1)^m\frac{a_{2m}(\nu)}{z^{2m}} - \sin(\mu)\!\!\sum_{m=0}^{M-1}(-1)^m\frac{a_{2m+1}(\nu)}{z^{2m+1}} \right) + R_{\nu,M}(z),
\label{eq:HankelExpansion}
\end{equation} 
where $\mu = z-(2\nu + 1)\pi/4$, $|R_{\nu,M}(z)|$ is the error term, $a_0(\nu) = 1$, and
\[
a_{m}(\nu) = \frac{(4\nu^2 - 1^2) ( 4\nu^2 - 3^2) \cdots (4\nu^2 - (2m-1)^2)}{m! 8^m}, \qquad m\geq 1.
\]

The first term in~\eqref{eq:HankelExpansion} shows that the leading asymptotic behavior of $J_\nu(z)$ 
as $z\rightarrow\infty$ is $\sqrt{2/(\pi z)} \cos(\mu)$, 
which is an oscillating trigonometric function that decays like $1/\sqrt{z}$. 
The first $2M$ terms show that, up to an error of $|R_{\nu,M}(z)|$, 
$J_\nu(z)$ can be written as a weighted sum of $\cos(\mu)z^{-2m-1/2}$ 
and $\sin(\mu)z^{-2m-3/2}$ for $0\leq m\leq M-1$.
Figure~\ref{fig:AsyDivergence} (left) shows the error term $|R_{0,M}(z)|$ on $z\in (0,50]$
for $M=1,2,3,4,5,7,10$. As expected, the asymptotic expansion becomes 
more accurate as $z$ increases. In particular, for sufficiently large $z$ 
the error term is negligible, i.e., $|R_{\nu,M}(z)|\leq \epsilon$. 

It is important to appreciate that~\eqref{eq:HankelExpansion} is an 
asymptotic expansion, as opposed to a convergent series, and 
does not converge pointwise to $J_\nu(z)$ as $M\rightarrow \infty$. 
In practice, increasing $M$ will eventually be detrimental and 
lead to severe numerical overflow issues.
Figure~\ref{fig:AsyDivergence} (right) shows the 
error $|R_{\nu,M}(10)|$ as $M\rightarrow \infty$ and $\nu=0,5,10,20$. 

Thus, the appropriate use of~\eqref{eq:HankelExpansion} is to 
select an $M\geq 1$ and ask: ``For what sufficiently large $z$ is the asymptotic 
expansion accurate?''. For example, from Figure~\ref{fig:AsyDivergence} (left) 
we observe that for $M=10$ we have $|R_{0,10}(z)|\leq 10^{-15}$ for any $z>17.8$ 
and hence, we may safely replace $J_0(z)$ by a weighted sum of $\cos(\mu)z^{-2m-1/2}$ 
and $\sin(\mu)z^{-2m-3/2}$ for $0\leq m\leq 9$ when $z>17.8$. 

\begin{figure} 
 \centering 
 \begin{minipage}{.49\textwidth} 
 \centering
  \begin{overpic}[width=\textwidth]{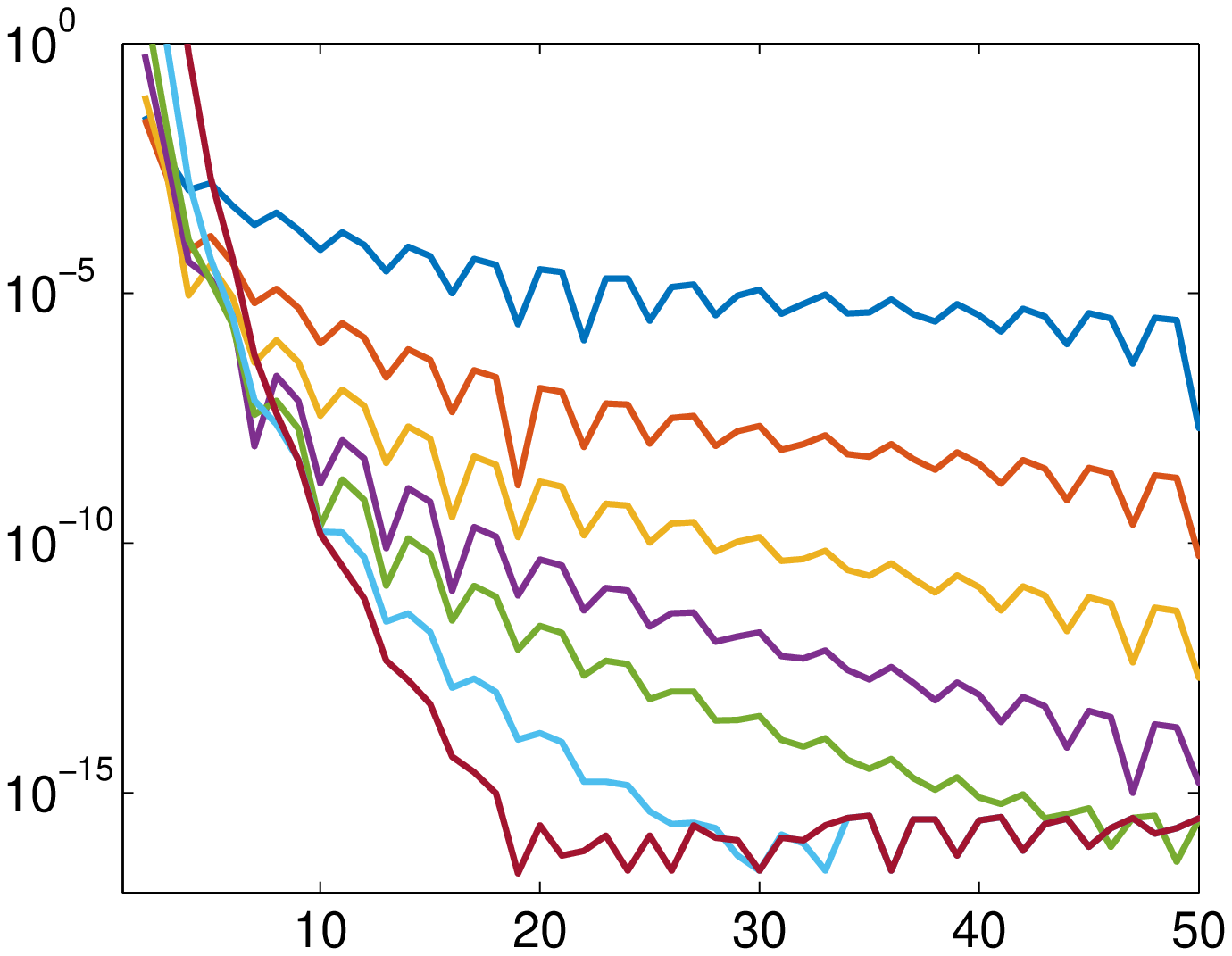}
  \put(50,0) {$z$}
 \put(-1.5,27) {\rotatebox{90}{$|R_{0,M}(z)|$}}
 \put(75,51.5) {\footnotesize{\rotatebox{-5}{$M=1$}}}
 \put(74,41) {\footnotesize{\rotatebox{-8}{$M=2$}}}
 \put(73,32) {\footnotesize{\rotatebox{-11}{$M=3$}}}
 \put(71,25.2) {\footnotesize{\rotatebox{-15}{$M=4$}}}
 \put(68,19) {\footnotesize{\rotatebox{-17}{$M=5$}}}
 \put(44,20) {\footnotesize{\rotatebox{-27}{$M=7$}}}
 \put(35.2,22.2) {\footnotesize{\rotatebox{-45}{$M=10$}}}
  \end{overpic}
 \end{minipage}
 \begin{minipage}{.49\textwidth} 
 \centering
  \begin{overpic}[width=\textwidth]{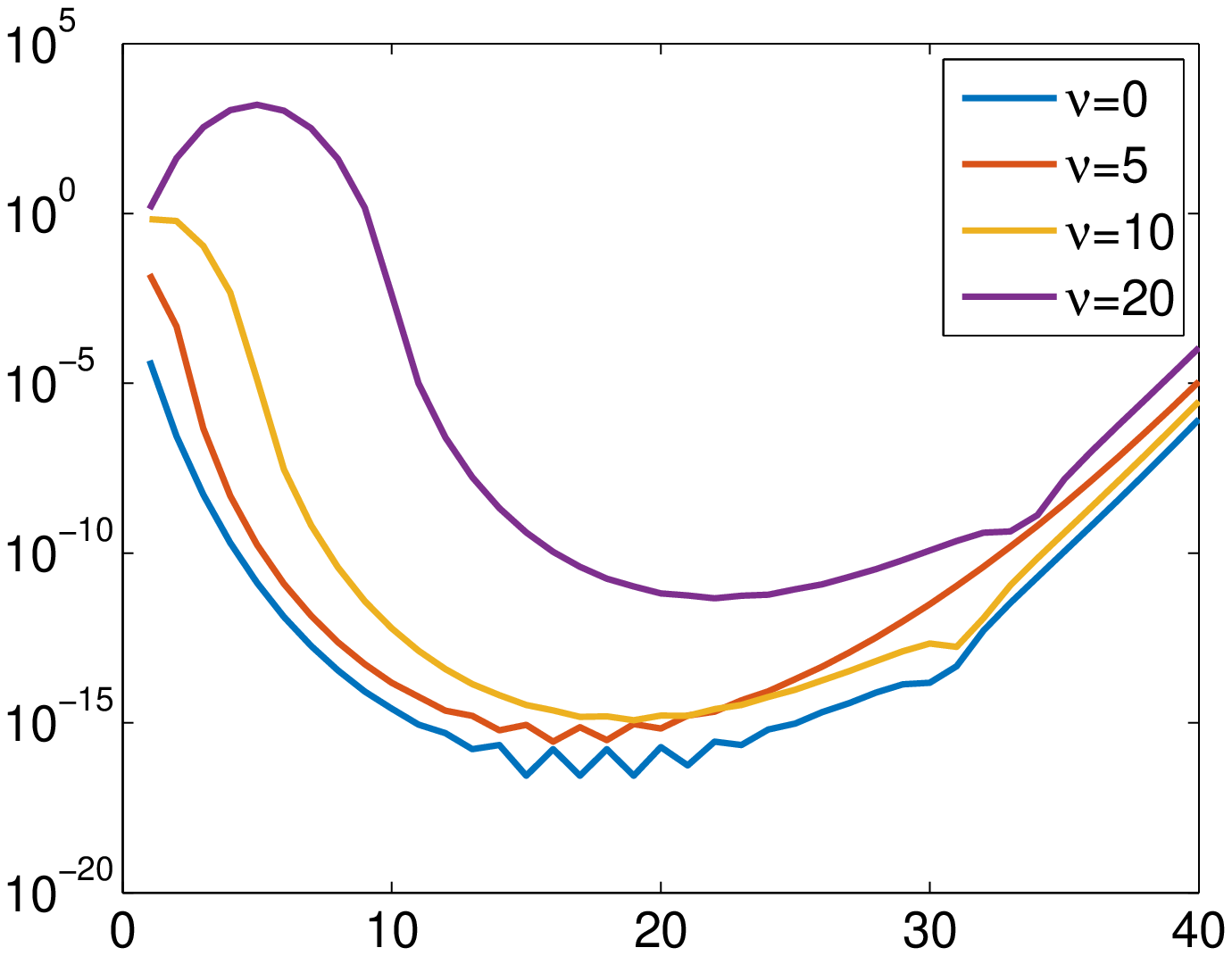}
  \put(50,0) {$M$}
  \put(-1.5,27) {\rotatebox{90}{$|R_{\nu,M}(10)|$}}
  \end{overpic}
 \end{minipage}
 \caption{Left: The absolute value of $R_{0,M}(z)$ in~\eqref{eq:HankelExpansion} 
 on $z\in(0,50]$ for $M=1,2,3,4,5,7,10$. For any $M\geq 1$ and sufficiently 
 large $z$, the asymptotic expansion is an accurate approximation. 
 Right: The absolute value of $R_{\nu,M}(10)$ for $1\leq M\leq 40$ and $\nu=0,5,10,20$.  
 For a fixed $z$, increasing $M$ may reduce the accuracy of the asymptotic expansion
 and will eventually lead to numerical overflow issues.}
 \label{fig:AsyDivergence}
\end{figure}

More generally, it is known that the error term $|R_{\nu,M}(z)|$ is bounded by the 
size of the first neglected terms~\cite[10.17(iii)]{NISTHandbook}. Since $|\cos(\omega)|\leq 1$ 
and $|\sin(\omega)|\leq 1$ we have
\begin{equation}
 | R_{\nu,M}(z) | \leq \left(\frac{2}{\pi z} \right)^{\!\!\frac{1}{2}}\left(\frac{|a_{2M}(\nu)|}{z^{2M}} + \frac{|a_{2M+1}(\nu)|}{z^{2M+1}} \right).
\label{eq:errorBound}
\end{equation}
This is essentially a sharp bound since for any $c>\sqrt{2}$
there is a $z>0$ such 
that $c|R_{\nu,M}(z)|$ is larger than the magnitude of the first neglected terms.\footnote{To see 
this pick $z$ of the form $(k/4+\nu/2+1/4)\pi$ for sufficiently large $k$.}

Let $s_{\nu,M}(\epsilon)$ be the smallest real number so that if
$z\geq s_{\nu,M}(\epsilon)$ then $| R_{\nu,M}(z) |\leq\epsilon$. 
From~\eqref{eq:errorBound} we can take $s_{\nu,M}(\epsilon)$ as the number 
that solves
\[
 \left(\frac{2}{\pi z} \right)^{\frac{1}{2}}\left(\frac{|a_{2M}(\nu)|}{z^{2M}} + \frac{|a_{2M+1}(\nu)|}{z^{2M+1}} \right) = \epsilon.
\]
In general, the equation above has no closed-form solution, but 
can be solved numerically by a 
fixed-point iteration. We start with the initial guess of 
$s_{\nu,M}^{(0)}(\epsilon) = 1$ and iterate as follows:
\[
 s_{\nu,M}^{(j+1)}(\epsilon)  = \left(\frac{\sqrt{2}\left(|a_{2M}(\nu)|+|a_{2M+1}(\nu)|/s_{\nu,M}^{(j)}(\epsilon) \right)}{\sqrt{\pi}\epsilon}\right)^{\frac{1}{2M+\frac{1}{2}}}, \qquad j\geq 0.
\]
We terminate this after four iterations and 
take $s_{\nu,M}(\epsilon)\approx s_{\nu,M}^{(4)}(\epsilon)$. Table~\ref{tab:sMTable}
gives the calculated values of $s_{\nu,M}(10^{-15})$ 
for $3\leq M\leq 12$ and $\nu = 0,1,2,10$. 

\begin{table} 
\centering
\begin{tabular}{cccccccccccc} 
\hline 
\rule{0pt}{3ex}$M$ \rule{0pt}{3ex}& 3& 4 & 5 & 6 & 7 & 8 & 9 & 10 & 11 & 12\\[2pt]
\hline
\rule{0pt}{3ex}$s_{0,M}$ \rule{0pt}{3ex}& 180.5 & 70.5 & 41.5 & 30.0 & 24.3 & 21.1 & 19.1 & 17.8 & 17.0 & 16.5 \\[2pt]
$s_{1,M}$ & 185.2 & 71.5 & 41.9 & 30.2 & 24.4 & 21.1 & 19.2 & 17.9 & 17.1 & 16.5 \\[2pt]
$s_{2,M}$ & 200.2 & 74.8 & 43.1 & 30.8 & 24.8 & 21.4 & 19.3 & 18.0 & 17.2 & 16.6 \\[2pt]
$s_{10,M}$ & 2330.7 & 500.0 & 149.0 & 64.6 & 41.4 & 31.4 & 26.0 & 22.9 & 20.9 & 19.6 \\[2pt]
\hline
\end{tabular}
\caption{The asymptotic expansion~\eqref{eq:HankelExpansion} is 
accurate up to an error of $\epsilon$ for any $z\geq s_{\nu,M}(\epsilon)$. 
Here, the values of $s_{\nu,M}(10^{-15})$ are tabulated for $3\leq M\leq 12$ 
and $\nu=0,1,2,10$.}
\label{tab:sMTable}
\end{table} 

\subsection{The Taylor series expansion of Bessel functions for small arguments}\label{sec:BesselTaylor}
For an integer $\nu$, the Taylor series 
expansion of $J_\nu(z)$ about $z=0$ is given by~\cite[(10.2.2)]{NISTHandbook} 
\begin{equation}
J_{\nu}(z) = \sum_{t=0}^{\infty} \frac{(-1)^t2^{-2t-\nu}}{t!(t+\nu)!}z^{2t+\nu}.
\label{eq:TaylorSeriesInfinite}
\end{equation}
In contrast to~\eqref{eq:HankelExpansion}, 
the Taylor series expansion converges pointwise to $J_\nu(z)$ for any $z$.
However, for practical application the infinite series in~\eqref{eq:TaylorSeriesInfinite} must still be 
truncated. 

Let $T\geq1$ be an integer and consider the truncated Taylor series expansion that is given by
\begin{equation} 
J_{\nu}(z) = \sum_{t=0}^{T-1} \frac{(-1)^t2^{-2t-\nu}}{t!(t+\nu)!}z^{2t+\nu} + E_{\nu,T}(z),
\label{eq:taylorSeries} 
\end{equation} 
where $|E_{\nu,T}(z)|$ is the error term. As $z\rightarrow 0$ the leading asymptotic 
behavior of $|E_{\nu,T}(z)|$ matches the order of the first neglected term and hence, 
$\lim_{z\rightarrow 0} |E_{\nu,T}(z)| = 0$. In particular, there is a 
real number $t_{\nu,T}(\epsilon)$ so that if $z\leq t_{\nu,T}(\epsilon)$ then
$|E_{\nu,T}(z)|\leq \epsilon$. 

To calculate the parameter 
$t_{\nu,T}(\epsilon)$ we solve the equation $E_{\nu,T}(z)=\epsilon$. Remarkably, 
an explicit closed-form expression for $E_{\nu,T}(z)$ is known~\cite{Wolfram_14_01}
and is given by
\[
 E_{\nu,T}(z) = \frac{(-1)^T 2^{-2T-\nu}z^{2T+\nu}}{(T+\nu)!T!}{}_1\!F_2\left(\begin{matrix}&1&\\T+1\!\!\!\!\!&
&\!\!\!\!\!T+\nu+1\end{matrix};-\frac{z^2}{4}\right)\approx \frac{(-1)^T 2^{-2T-\nu}z^{2T+\nu}}{(T+\nu)!T!},
\]
where ${}_1\!F_2$ is the generalized hypergeometric function that we 
approximate by $1$ since we are considering $z$ to be small~\cite[(16.2.1)]{NISTHandbook}. 
Solving $E_{\nu,T}(z)=\epsilon$ we find that
\[
 t_{\nu,T}(\epsilon) \approx \left(2^{2T+\nu}(T+\nu)!T!\epsilon\right)^{\frac{1}{2T+\nu}}.
\]
Table~\ref{tab:tTTable} gives the calculated values of $t_{\nu,T}(10^{-15})$ 
for $1\leq T\leq 9$ and $\nu = 0,1,2,10$. 
\begin{table} 
\centering
\begin{tabular}{ccccccccccc} 
\hline 
\rule{0pt}{3ex}$T$ \rule{0pt}{3ex}& 1 & 2 & 3& 4 & 5 & 6 & 7 & 8 & 9 \\[2pt]
\hline
\rule{0pt}{3ex}$t_{0,T}$ \rule{0pt}{3ex}&  $6\times 10^{-8}$ & 0.001 & 0.011 & 0.059 & 0.165 & 0.337 & 0.573 & 0.869 & 1.217 \\[2pt]
$t_{1,T}$ & $3\times 10^{-5}$ & 0.003 & 0.029 & 0.104 & 0.243 & 0.449 & 0.716 & 1.039 & 1.411 \\[2pt]
$t_{2,T}$ & 0.001 & 0.012 & 0.061 & 0.168 & 0.341 & 0.579 & 0.876 & 1.225 & 1.618 \\[2pt]
$t_{10,T}$ & 0.484 & 0.743 & 1.058 & 1.420 & 1.823 & 2.262 & 2.733 & 3.230 & 3.750 \\[2pt]
\hline
\end{tabular}
\caption{The truncated Taylor series expansion~\eqref{eq:taylorSeries} is 
accurate up to an error of $\epsilon$ for any $z\leq t_{\nu,T}(\epsilon)$. Here, 
the approximate values of $t_{\nu,T}(10^{-15})$ are tabulated for $1\leq T\leq 9$ and $\nu=0,1,2,10$.}
\label{tab:tTTable}
\end{table} 

\subsection{The Neumann addition formula for perturbed arguments}
The Neumann addition formula expresses $J_{\nu}(z+\delta z)$ as an 
infinite sum of products of Bessel functions of the form 
$J_{\nu-s}(z)J_s(\delta z)$ for $s\in\mathbb{Z}$. It is given by~\cite[(10.23.2)]{NISTHandbook}
\begin{equation}
 J_{\nu}(z + \delta z) = \sum_{s=-\infty}^\infty J_{\nu-s}(z)J_{s}(\delta z).
\label{eq:NeumannAdditionFormula}
\end{equation}
Here, we wish to truncate~\eqref{eq:NeumannAdditionFormula} and use it 
as a numerical approximation to $J_{\nu}(z+\delta z)$. Fortunately, 
when $|\delta z|$ is small (so that $z+\delta z$ can be considered as a 
perturbation of $z$)
the Neumann addition formula is a rapidly converging 
series for $J_\nu(z+\delta z)$. 
\begin{lemma} 
 Let $\nu$ be an integer, $z>0$, and $|\delta z|< e^{-1}$. Then, for $K\geq 1$ we have 
 \begin{equation}
  \left| J_{\nu}(z + \delta z) - \sum_{s=-K+1}^{K-1} J_{\nu-s}(z)J_{s}(\delta z) \right| \leq 5.2\left(\frac{e\,|\delta z|}{2}\right)^K,
 \label{eq:NeumannBound}
 \end{equation}
 where $e\approx 2.71828$ is Euler's number.
 \label{lem:NeumannBound}
\end{lemma}
\begin{proof} 
Let $\nu$ be an integer, $z>0$, $|\delta z|< e^{-1}$, and $K\geq 1$.
 Denote the left-hand side of the inequality in~\eqref{eq:NeumannBound} by $|A_{\nu,K}(z,\delta z)|$ so that  
 \[
  A_{\nu,K}(z,\delta z) = \sum_{s=-\infty}^{-K} J_{\nu-s}(z)J_{s}(\delta z) + \sum_{s=K}^\infty J_{\nu-s}(z)J_{s}(\delta z).
 \]
 Since $|J_\nu(z)|\leq 1$~\cite[(10.14.1)]{NISTHandbook} and $J_{-\nu}(z) = (-1)^\nu J_{\nu}(z)$~\cite[(10.4.1)]{NISTHandbook}, we have 
\begin{equation}
 \left|A_{\nu,K}(z,\delta z)\right| \leq \sum_{s=-\infty}^{-K} \left|J_{s}(\delta z)\right| + \sum_{s=K}^\infty \left|J_{s}(\delta z)\right|\leq 2\sum_{s=K}^\infty\left|J_{s}(\delta z)\right|.
\label{eq:Abound}
\end{equation}
By Kapteyn's inequality~\cite[(10.14.8)]{NISTHandbook} we can bound $|J_s(\delta z)|$ from above as follows: 
\begin{equation}
 |J_{s}(\delta z)| = |J_{s}(s(\delta z/s))|\leq \frac{\left|\delta z/s\right|^se^{s(1-\left|\delta z/s\right|^2)^{\frac{1}{2}}}}{\left(1+(1-\left|\delta z/s\right|^2)^{\frac{1}{2}}\right)^s}\leq \left(\frac{e\,|\delta z|}{2s}\right)^s ,
\label{eq:Jbound}
\end{equation} 
where the last inequality comes from $e^x\leq e(1+x)/2$ for $0\leq x\leq 1$. Therefore, by substituting~\eqref{eq:Jbound} into~\eqref{eq:Abound} we find that
\[
 \left|A_{\nu,K}(z,\delta z)\right| \leq 2\sum_{s=K}^\infty \left(\frac{e\,|\delta z|}{2s}\right)^s \leq 2\sum_{s=K}^\infty s^{-s} \sum_{s=K}^\infty \left(\frac{e\,|\delta z|}{2}\right)^s \leq 5.2\left(\frac{e\,|\delta z|}{2}\right)^K, 
\]
where we used $\sum_{s=1}^\infty s^{-s}\leq 1.3$ and $\sum_{s=0}^\infty (e|\delta z|/2)^s = 1/(1-e|\delta z|/2)\leq 2$. 
\end{proof}

Lemma~\ref{lem:NeumannBound} shows that $\sum_{s=-K+1}^{K-1} J_{\nu-s}(z)J_s(\delta z)$
approximates $J_\nu(z + \delta z)$, up to an error of $\epsilon$, provided that
$5.2(e|\delta z|/2)^K\leq \epsilon$. Equivalently, for any 
\begin{equation}
 |\delta z| \leq \frac{2}{e} \left(\frac{\epsilon}{5.2} \right)^{\frac{1}{K}}.
\label{eq:NeumannInequality}
\end{equation}
Table~\ref{tab:NeumannTable} gives the values of $2(\epsilon/5.2)^{1/K}/e$ for 
$\epsilon = 10^{-15}$.  
\begin{table} 
\centering
\begin{tabular}{ccccccccccc} 
\hline 
\rule{0pt}{3ex}$K$ \rule{0pt}{3ex}& 3& 4 & 5 & 6 & 7 & 8 & 9 & 10 \\[2pt]
\hline
\rule{0pt}{3ex}\eqref{eq:NeumannInequality}\rule{0pt}{3ex}& $4\times 10^{-6}$ & $9\times10^{-5}$ & 0.001 & 0.002 & 0.004 & 0.008 & 0.013 & 0.020\\[2pt]
\hline
\end{tabular}
\caption{Lemma~\ref{lem:NeumannBound} shows that $\sum_{s=-K+1}^{K-1}J_{\nu-s}(z)J_s(\delta z)$ is 
an approximation to $J_{\nu}(z+\delta z)$, up to an error of $\epsilon$, 
for any $|\delta z|\leq 2(\epsilon/5.2)^{1/K}/e$. Here, the values of 
$2(\epsilon/5.2)^{1/K}/e$ are tabulated for $\epsilon=10^{-15}$ and $3\leq K\leq 10$.}
\label{tab:NeumannTable}
\end{table} 

At first it may not be clear why Lemma~\ref{lem:NeumannBound} is useful 
for practical computations since a single evaluation of 
$J_{\nu}(z + \delta z)$ is replaced by a seemingly trickier sum of 
products of Bessel functions. However, in sections~\ref{sec:FourierBessel} 
and~\ref{sec:DiscreteHankelTransform} the truncated 
Neumann addition formula will be the approximation that will allow 
us to evaluate Fourier--Bessel expansions and compute the discrete Hankel transform.
The crucial observation is that the positive roots of $J_0(z)$ can be regarded as 
a perturbed equally-spaced grid. 

\subsubsection{The roots of $\mathbf{J_0(z)}$ as a perturbed equally-spaced grid}\label{sec:BesselRootsPerturbed}
Let $j_{0,n}$ denote the $n$th positive root of $J_0(z)$. 
We know from~\eqref{eq:HankelExpansion} that the leading asymptotic behavior 
of $J_0(z)$ is $\sqrt{2/(\pi z)}\cos(z-\pi/4)$ for large $z$. 
Therefore, since $\cos((n-1/2)\pi)=0$ for $n\geq 1$ the leading asymptotic behavior 
of $j_{0,n}$ is $(n-1/4)\pi$ for large $n$. 
Similarly, the leading asymptotic behavior of 
$j_{0,k}/j_{0,N+1}$ is $(k-1/4)/(N+3/4)$ for large $k$ and $N$.
The next lemma shows that $j_{0,1},\ldots,j_{0,N}$ and the ratios 
$j_{0,1}/j_{0,N+1},\ldots,j_{0,N}/j_{0,N+1}$ can be regarded as perturbed
equally-spaced grids.

\begin{lemma} 
Let $j_{0,n}$ denote the $n$th positive root of $J_0(z)$. Then, 
 \[
  j_{0,n} = \left(n-\frac{1}{4}\right)\pi + b_n, \qquad 0\leq b_n \leq \frac{1}{8(n-\frac{1}{4})\pi},
 \]
 and 
 \[
  \frac{j_{0,k}}{j_{0,N+1}} = \frac{k-\frac{1}{4}}{N+\frac{3}{4}} + e_{k,N}, \qquad |e_{k,N}|\leq \frac{1}{8(N+\frac{3}{4})(k-\frac{1}{4})\pi^2}. 
 \]
 \label{lem:BesselInequalities}
\end{lemma}
\begin{proof} 
 The bounds on $b_n$ are given in~\cite[Thm.~3]{Hethcote_70_01}. For the 
 bound on $|e_{k,N}|$ we have (since $j_{0,k} = (k-1/4)\pi+b_k$ and $j_{0,N+1} = (N+3/4)\pi+b_{N+1}$)
 \[
  |e_{k,N}| = \left|\frac{j_{0,k}}{j_{0,N+1}} - \frac{(k-\frac{1}{4})}{(N+\frac{3}{4})}\right| = \left|\frac{b_k(N+\frac{3}{4})-b_{N+1}(k-\frac{1}{4})}{((N+\frac{3}{4})\pi+b_{N+1})(N+\frac{3}{4})}\right|.
 \]
The result follows from $b_{N+1}\geq 0$ and $|b_k(N+3/4)-b_{N+1}(k-1/4)|\leq (N+3/4)/(8(k-1/4)\pi)$.
\end{proof}

Lemma~\ref{lem:BesselInequalities} shows that
we can consider $j_{0,1},\ldots,j_{0,N}$ as a small perturbation of $3\pi/4,7\pi/4,\ldots,(N-1/4)\pi$ 
and that we can consider $j_{0,1}/j_{0,N+1},\ldots,j_{0,N}/j_{0,N+1}$ as a small perturbation 
of $3/(4N+3),7/(4N+3),\ldots,(4N-1)/(4N+3)$. This is an important observation for 
sections~\ref{sec:FourierBessel} and~\ref{sec:DiscreteHankelTransform}.  

\section{Fast evaluation schemes for Schl\"{o}milch expansions}\label{sec:Schlomilch}
A Schl\"{o}milch expansion of a function $f:[0,1]\rightarrow \mathbb{C}$ takes 
the form~\cite{Schlomilch_1846_01}:
\begin{equation}
 f(r) = \sum_{n=1}^N c_n J_\nu(n\pi r). 
\label{eq:Schlomilch}
\end{equation}
Here, we take $\nu\geq 0$ to be an integer and develop a fast $\mathcal{O}(N(\log N)^2/\log\!\log N)$
algorithm for evaluating Schl\"{o}milch expansions
at $r_1,\ldots,r_N$, where $r_k = k/N$. The evaluation of~\eqref{eq:Schlomilch} 
at $r_1,\ldots,r_N$ is equivalent to computing the matrix-vector 
product $\underline{f} = \mathbf{J}_\nu(\, \underline{r}\,\underline{\omega}^\intercal )\underline{c}$, i.e.,
\begin{equation}
 \begin{pmatrix} 
   f_1\\\vdots\\f_{N}
 \end{pmatrix} = 
\underbrace{
 \begin{pmatrix} 
  J_{\nu}( r_1\omega_1  ) & \cdots & J_{\nu}( r_1\omega_N ) \\
  \vdots & \ddots & \vdots \\ 
  J_{\nu}( r_N\omega_1) & \cdots & J_{\nu}( r_N\omega_N ) 
 \end{pmatrix}}_{ \mathbf{J}_\nu(\, \underline{r}\,\underline{\omega}^\intercal )}
  \begin{pmatrix} 
   c_1\\ \vdots \\ c_N
 \end{pmatrix},
\label{eq:SchlomilchMatrix}
\end{equation}
where $f_k = f(r_k)$ and $\omega_n = n\pi$. 
The $(k,n)$ entry of $\mathbf{J}_\nu(\, \underline{r}\,\underline{\omega}^\intercal )$
is $J_{\nu}(kn\pi/N)$, which takes a similar form to the $(k+1,n+1)$ entry 
of $C^I_{N+1}$, $\cos(kn\pi/N)$, and the $(k,n)$ entry of $S^I_{N-1}$, $\sin(kn\pi/N)$, see~\eqref{eq:DCTs}. 
We will carefully use the asymptotic expansion in~\eqref{eq:HankelExpansion} 
to approximate $\mathbf{J}_\nu(\, \underline{r}\,\underline{\omega}^\intercal )$
by a weighted sum of $C^I_{N+1}$ and $S^I_{N-1}$.

\subsection{A simple and fast evaluation scheme for Schl\"{o}milch expansions}\label{sec:singlePartition}
To derive a fast matrix-vector product for 
$\mathbf{J}_\nu(\, \underline{r}\,\underline{\omega}^\intercal )$,
we first substitute $z = r_k\omega_n$ into the asymptotic expansion~\eqref{eq:HankelExpansion}. 
By the trigonometric addition formula we have
\[
 \cos\!\left( r_k\omega_n -(2\nu + 1)\frac{\pi}{4} \right) = \cos\!\left(\frac{kn\pi}{N}\right)\cos\!\left((2\nu+1)\frac{\pi}{4}\right) + \sin\!\left(\frac{kn\pi}{N}\right)\sin\!\left((2\nu+1)\frac{\pi}{4}\right).
\]
As a matrix decomposition we can write this as the following sum of $C_{N+1}^I$ and $S_{N-1}^I$:
\begin{equation}
 \cos\left( \, \underline{r}\,\underline{\omega}^\intercal -(2\nu + 1)\frac{\pi}{4} \right) = d_1Q^\intercal C_{N+1}^{I}Q + d_2\begin{bmatrix} S_{N-1}^{I} &\!\!\!\! \underline{0}\\ \underline{0}^\intercal &\!\!\!\! 0\end{bmatrix},
\label{eq:cos}
\end{equation} 
where $\underline{0}$ is a zero column vector, 
$Q = \left[\,\underline{0}\,\,\,I_{N}\right]^\intercal$, $d_1 = \cos((2\nu+1)\pi/4)$, and $d_2 = \sin((2\nu+1)\pi/4)$. Similarly, for sine we have
\begin{equation}
 \sin\left( \, \underline{r}\,\underline{\omega}^\intercal -(2\nu + 1)\frac{\pi}{4} \right) = - d_2 Q^\intercal C_{N+1}^{I}Q  +  d_1\begin{bmatrix} S_{N-1}^{I} &\!\!\!\! \underline{0}\\ \underline{0}^\intercal &\!\!\!\! 0\end{bmatrix}.
\label{eq:sin}
\end{equation}

Thus, by substituting~\eqref{eq:cos} and~\eqref{eq:sin} 
into~\eqref{eq:HankelExpansion} with $z = \underline{r}\,\underline{\omega}^\intercal$ we obtain a
matrix decomposition for $\mathbf{J}_\nu(\, \underline{r}\,\underline{\omega}^\intercal )$,
\begin{equation}
\begin{aligned}
 \mathbf{J}_\nu(\, \underline{r}&\,\underline{\omega}^\intercal ) = \sum_{m=0}^{M-1}\frac{(-1)^ma_{2m}(\nu)}{\sqrt{\pi/2}}D_{\underline{r}}^{-2m-\frac{1}{2}}\!\left( d_1Q^\intercal C_{N+1}^{I}Q + d_2\! \begin{bmatrix} S_{N-1}^{I} &\!\!\!\! \underline{0}\\ \underline{0}^\intercal &\!\!\!\! 0\end{bmatrix} \right)D_{\underline{\omega}}^{-2m-\frac{1}{2}}\! \\   
    &- \sum_{m=0}^{M-1}\frac{(-1)^ma_{2m+1}(\nu)}{\sqrt{\pi/2}} D_{\underline{r}}^{-2m-\frac{3}{2}}\!\left(-d_2Q^\intercal C_{N+1}^{I}Q+d_1\!\begin{bmatrix} S_{N-1}^{I} &\!\!\!\! \underline{0}\\ \underline{0}^\intercal &\!\!\!\! 0\end{bmatrix}\right)D_{\underline{\omega}}^{-2m-\frac{3}{2}}\! \\ 
    & \hspace{9cm}+ \mathbf{R}_{\nu,M}(\underline{r}\,\underline{\omega}^\intercal), 
\end{aligned}
\label{eq:AsymptoticDecomposition}
\end{equation}
where we used the fact that $\underline{r}\,\underline{\omega}^\intercal$ is a rank $1$ matrix.
We can write~\eqref{eq:AsymptoticDecomposition} more conveniently as 
\begin{equation}
 \mathbf{J}_\nu(\, \underline{r}\,\underline{\omega}^\intercal ) = \mathbf{J}_{\nu,M}^{\rm ASY}(\, \underline{r}\,\underline{\omega}^\intercal )  + \mathbf{R}_{\nu,M}(\, \underline{r}\,\underline{\omega}^\intercal ).
\label{eq:asyR}
\end{equation}

One could now compute $\mathbf{J}_{\nu,M}(\, \underline{r}\,\underline{\omega}^\intercal )\underline{c}$ by first computing the vector $\mathbf{J}_{\nu,M}^{\rm ASY}(\, \underline{r}\,\underline{\omega}^\intercal )\underline{c}$
and then correcting it by $\mathbf{R}_{\nu,M}(\, \underline{r}\,\underline{\omega}^\intercal )\underline{c}$. 
The matrix-vector product $\mathbf{J}_{\nu,M}^{\rm ASY}(\, \underline{r}\,\underline{\omega}^\intercal )\underline{c}$ can be computed in $\mathcal{O}(N\log N)$ operations
since the expression in~\eqref{eq:AsymptoticDecomposition} decomposes $\mathbf{J}_{\nu,M}^{\rm ASY}(\, \underline{r}\,\underline{\omega}^\intercal )$ as a weighted
sum of diagonal matrices and $2M$ DCTs and DSTs.

While 
this does lead to an $\mathcal{O}(N\log N)$ evaluation scheme for Schl\"{o}milch expansions, it is numerically unstable 
because the asymptotic expansion is employed for entries $J_\nu(r_k\omega_n)$ for which 
$r_k\omega_n$ is small (see Section~\ref{sec:asymptoticExpansion}). Numerical overflow issues and severe cancellation errors plague this approach. Figure~\ref{fig:partition} illustrates 
the entries that cause the most numerical problems. 
We must be more careful
and only employ the asymptotic expansion for the $(k,n)$ 
entry of $\mathbf{J}_{\nu,M}(\, \underline{r}\,\underline{\omega}^\intercal)$ if
$|R_{\nu,M}(r_k\omega_n)|\leq \epsilon$.

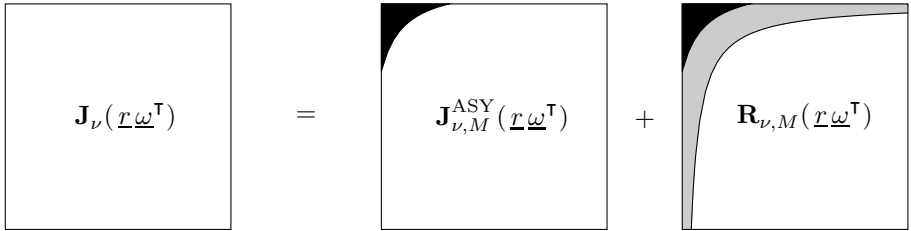
\begin{figure} 
 \centering
 \begin{tikzpicture}
\draw (0,0) -- (3,0) -- (3,3) -- (0,3) -- (0,0);
\draw (5,0) -- (8,0) -- (8,3) -- (5,3) -- (5,0);
\node[anchor=west] at (.8,1.5) (l1) {$\mathbf{J}_\nu(\, \underline{r}\,\underline{\omega}^\intercal )$};
\node[anchor=west] at (5.6,1.5) (l1) {$\mathbf{J}_{\nu,M}^{\rm ASY}(\, \underline{r}\,\underline{\omega}^\intercal )$};
\node[anchor=west] at (9.6,1.5) (l2) {$\mathbf{R}_{\nu,M}(\, \underline{r}\,\underline{\omega}^\intercal )$};
\node[anchor=north] at (8.5,1.7) (l2) {$+$};
\node[anchor=north] at (4,1.7) (l2) {$=$};
\draw[fill,black!20] (9,0)--(0.200*3/5+9,0.000*3/5)--(0.216*3/5+9,0.366*3/5)--(0.221*3/5+9,0.476*3/5)--(0.226*3/5+9,0.581*3/5)--(0.232*3/5+9,0.682*3/5)--(0.237*3/5+9,0.778*3/5)--(0.242*3/5+9,0.870*3/5)--(0.247*3/5+9,0.957*3/5)--(0.253*3/5+9,1.042*3/5)--(0.258*3/5+9,1.122*3/5)--(0.263*3/5+9,1.200*3/5)--(0.268*3/5+9,1.275*3/5)--(0.274*3/5+9,1.346*3/5)--(0.279*3/5+9,1.415*3/5)--(0.284*3/5+9,1.481*3/5)--(0.289*3/5+9,1.545*3/5)--(0.295*3/5+9,1.607*3/5)--(0.300*3/5+9,1.667*3/5) --(0.385*3/5+9,2.400*3/5)--(0.513*3/5+9,3.050*3/5)--(0.641*3/5+9,3.440*3/5)--(0.769*3/5+9,3.700*3/5)--(0.897*3/5+9,3.886*3/5)--(1.026*3/5+9,4.025*3/5)--(1.154*3/5+9,4.133*3/5)--(1.282*3/5+9,4.220*3/5)--(1.410*3/5+9,4.291*3/5)--(1.538*3/5+9,4.350*3/5)--(1.667*3/5+9,4.400*3/5)--(1.795*3/5+9,4.443*3/5)--(1.923*3/5+9,4.480*3/5)--(2.051*3/5+9,4.513*3/5)--(2.179*3/5+9,4.541*3/5)--(2.308*3/5+9,4.567*3/5)--(2.436*3/5+9,4.589*3/5)--(2.564*3/5+9,4.610*3/5)--(2.692*3/5+9,4.629*3/5)--(2.821*3/5+9,4.645*3/5)--(2.949*3/5+9,4.661*3/5)--(3.077*3/5+9,4.675*3/5)--(3.205*3/5+9,4.688*3/5)--(3.333*3/5+9,4.700*3/5)--(3.462*3/5+9,4.711*3/5)--(3.590*3/5+9,4.721*3/5)--(3.718*3/5+9,4.731*3/5)--(3.846*3/5+9,4.740*3/5)--(3.974*3/5+9,4.748*3/5)--(4.103*3/5+9,4.756*3/5)--(4.231*3/5+9,4.764*3/5)--(4.359*3/5+9,4.771*3/5)--(4.487*3/5+9,4.777*3/5)--(4.615*3/5+9,4.783*3/5)--(4.744*3/5+9,4.789*3/5)--(4.872*3/5+9,4.795*3/5)--(5.000*3/5+9,4.800*3/5)--(12,3)--(9,3)--(9,0);
\draw (0.200*3/5+9,0.000*3/5)--(0.216*3/5+9,0.366*3/5)--(0.221*3/5+9,0.476*3/5)--(0.226*3/5+9,0.581*3/5)--(0.232*3/5+9,0.682*3/5)--(0.237*3/5+9,0.778*3/5)--(0.242*3/5+9,0.870*3/5)--(0.247*3/5+9,0.957*3/5)--(0.253*3/5+9,1.042*3/5)--(0.258*3/5+9,1.122*3/5)--(0.263*3/5+9,1.200*3/5)--(0.268*3/5+9,1.275*3/5)--(0.274*3/5+9,1.346*3/5)--(0.279*3/5+9,1.415*3/5)--(0.284*3/5+9,1.481*3/5)--(0.289*3/5+9,1.545*3/5)--(0.295*3/5+9,1.607*3/5)--(0.300*3/5+9,1.667*3/5) --(0.385*3/5+9,2.400*3/5)--(0.513*3/5+9,3.050*3/5)--(0.641*3/5+9,3.440*3/5)--(0.769*3/5+9,3.700*3/5)--(0.897*3/5+9,3.886*3/5)--(1.026*3/5+9,4.025*3/5)--(1.154*3/5+9,4.133*3/5)--(1.282*3/5+9,4.220*3/5)--(1.410*3/5+9,4.291*3/5)--(1.538*3/5+9,4.350*3/5)--(1.667*3/5+9,4.400*3/5)--(1.795*3/5+9,4.443*3/5)--(1.923*3/5+9,4.480*3/5)--(2.051*3/5+9,4.513*3/5)--(2.179*3/5+9,4.541*3/5)--(2.308*3/5+9,4.567*3/5)--(2.436*3/5+9,4.589*3/5)--(2.564*3/5+9,4.610*3/5)--(2.692*3/5+9,4.629*3/5)--(2.821*3/5+9,4.645*3/5)--(2.949*3/5+9,4.661*3/5)--(3.077*3/5+9,4.675*3/5)--(3.205*3/5+9,4.688*3/5)--(3.333*3/5+9,4.700*3/5)--(3.462*3/5+9,4.711*3/5)--(3.590*3/5+9,4.721*3/5)--(3.718*3/5+9,4.731*3/5)--(3.846*3/5+9,4.740*3/5)--(3.974*3/5+9,4.748*3/5)--(4.103*3/5+9,4.756*3/5)--(4.231*3/5+9,4.764*3/5)--(4.359*3/5+9,4.771*3/5)--(4.487*3/5+9,4.777*3/5)--(4.615*3/5+9,4.783*3/5)--(4.744*3/5+9,4.789*3/5)--(4.872*3/5+9,4.795*3/5)--(5.000*3/5+9,4.800*3/5);
\draw[fill,black] (5,3*3/5+.3)--(0.5*3/5+4.7,3*3/5+.3)--(0.513*3/5+4.7,3.050*3/5+.3)--(0.641*3/5+4.7,3.440*3/5+.3)--(0.769*3/5+4.7,3.700*3/5+.3)--(0.897*3/5+4.7,3.886*3/5+.3)--(1.026*3/5+4.7,4.025*3/5+.3)--(1.154*3/5+4.7,4.133*3/5+.3)--(1.282*3/5+4.7,4.220*3/5+.3)--(1.410*3/5+4.7,4.291*3/5+.3)--(1.538*3/5+4.7,4.350*3/5+.3)--(1.667*3/5+4.7,4.400*3/5+.3)--(1.795*3/5+4.7,4.443*3/5+.3)--(1.923*3/5+4.7,4.480*3/5+.3)--(2*3/5+4.7,4.5*3/5+.3)--(2*3/5+4.7,3)--(5,3)--(5,3*3/5+.3);
\draw[fill,black] (5+4,3*3/5+.3)--(0.5*3/5+4.7+4,3*3/5+.3)--(0.513*3/5+4.7+4,3.050*3/5+.3)--(0.641*3/5+4.7+4,3.440*3/5+.3)--(0.769*3/5+4.7+4,3.700*3/5+.3)--(0.897*3/5+4.7+4,3.886*3/5+.3)--(1.026*3/5+4.7+4,4.025*3/5+.3)--(1.154*3/5+4.7+4,4.133*3/5+.3)--(1.282*3/5+4.7+4,4.220*3/5+.3)--(1.410*3/5+4.7+4,4.291*3/5+.3)--(1.538*3/5+4.7+4,4.350*3/5+.3)--(1.667*3/5+4.7+4,4.400*3/5+.3)--(1.795*3/5+4.7+4,4.443*3/5+.3)--(1.923*3/5+4.7+4,4.480*3/5+.3)--(2*3/5+4.7+4,4.5*3/5+.3)--(2*3/5+4.7+4,3)--(5+4,3)--(5+4,3*3/5+.3);
\draw (9,0) -- (12,0) -- (12,3) -- (9,3) -- (9,0);
\end{tikzpicture}
\caption{The cause of numerical instability in the matrix decomposition~\eqref{eq:asyR}. 
The entries in the black region of $\mathbf{J}_{\nu,M}^{\rm ASY}(\, \underline{r}\,\underline{\omega}^\intercal )$ and $\mathbf{R}_{\nu,M}(\, \underline{r}\,\underline{\omega}^\intercal )$ are large, of similar magnitude, and 
of opposite sign, resulting in numerical overflow issues and severe cancellation errors when 
added together. The nonwhite region of $\mathbf{R}_{\nu,M}(\, \underline{r}\,\underline{\omega}^\intercal )$ contains all the entries of $\mathbf{R}_{\nu,M}(\, \underline{r}\,\underline{\omega}^\intercal )$ that are larger 
in magnitude than $\epsilon$. The asymptotic expansion should only be employed when
$|R_{\nu,M}(r_k\omega_n)|\leq \epsilon$.}
\label{fig:partition}
\end{figure}

By Section~\ref{sec:asymptoticExpansion} we know that if $z \geq s_{\nu,M}(\epsilon)$ then $|R_{\nu,M}(z)|\leq \epsilon$.  
Therefore, we can use~\eqref{eq:asyR} for the $(k,n)$ entry of $\mathbf{J}_\nu(\, \underline{r}\,\underline{\omega}^\intercal )$ provided that
$r_k\omega_n = nk\pi/N\geq s_{\nu,M}(\epsilon)$. This is guaranteed, for instance, 
when
\begin{equation}
\min(k,n) \geq \lceil \alpha N^{\frac{1}{2}} \rceil,
\label{eq:criterion}
\end{equation}
where $\alpha = (s_{\nu,M}(\epsilon)/\pi)^{1/2}$. Therefore, we can 
be more careful by taking the $N\times N$ diagonal 
matrix $Q_1$ with $(Q_1)_{ii}=1$ for $i\geq \lceil \alpha N^{\frac{1}{2}} \rceil$ 
and $0$ otherwise, and compute $\mathbf{J}_\nu(\, \underline{r}\,\underline{\omega}^\intercal )\underline{c}$ 
using the following matrix decomposition:
\begin{equation}
  \mathbf{J}_\nu(\, \underline{r}\,\underline{\omega}^\intercal ) = Q_1\mathbf{J}_{\nu,M}^{\rm ASY}(\, \underline{r}\,\underline{\omega}^\intercal )Q_1 + Q_1\mathbf{R}_{\nu,M}(\, \underline{r}\,\underline{\omega}^\intercal )Q_1 + \mathbf{J}_\nu^{\rm EVAL}(\, \underline{r}\,\underline{\omega}^\intercal ),
\label{eq:StableMatrixDecomposition}
\end{equation}
where $\mathbf{J}_\nu^{\rm EVAL}(\, \underline{r}\,\underline{\omega}^\intercal )$ is the matrix 
\[
 (\mathbf{J}_\nu^{\rm EVAL}(\, \underline{r}\,\underline{\omega}^\intercal ))_{kn} = \begin{cases} J_\nu(r_k\omega_n), & \min(k,n)<\lceil \alpha N^{\frac{1}{2}}\rceil, \\ 0, & \text{otherwise}. \end{cases}
\]

If the $(k,n)$ entry of $Q_1\mathbf{J}_{\nu,M}^{\rm ASY}(\, \underline{r}\,\underline{\omega}^\intercal )Q_1$ 
is nonzero, then $nk\pi/N\geq s_{\nu,M}(\epsilon)$ and hence,~\eqref{eq:StableMatrixDecomposition} 
only approximates $J_{\nu}(r_k\omega_n)$ 
by a weighted sum of trigonometric functions when it is safe to do so. Figure~\ref{fig:decomposition} shows how the criterion~\eqref{eq:criterion}
partitions the entries of $\mathbf{J}_\nu(\, \underline{r}\,\underline{\omega}^\intercal )$. 
A stable $\mathcal{O}(N^{3/2})$ algorithm for evaluating Schl\"{o}milch expansions~\eqref{eq:Schlomilch} 
follows, as we will now explain.

\begin{figure} 
 \centering
 \begin{tikzpicture}
 \draw[fill=black!20] (0,0) -- (1,0) -- (1,4) -- (5,4) -- (5,5) -- (0,5) -- (0,0);
  \draw (0,0) -- (5,0) -- (5,5) -- (0,5) -- (0,0);
\draw (0.200,0.000)--(0.216,0.366)--(0.221,0.476)--(0.226,0.581)--(0.232,0.682)--(0.237,0.778)--(0.242,0.870)--(0.247,0.957)--(0.253,1.042)--(0.258,1.122)--(0.263,1.200)--(0.268,1.275)--(0.274,1.346)--(0.279,1.415)--(0.284,1.481)--(0.289,1.545)--(0.295,1.607)--(0.300,1.667) --(0.385,2.400)--(0.513,3.050)--(0.641,3.440)--(0.769,3.700)--(0.897,3.886)--(1.026,4.025)--(1.154,4.133)--(1.282,4.220)--(1.410,4.291)--(1.538,4.350)--(1.667,4.400)--(1.795,4.443)--(1.923,4.480)--(2.051,4.513)--(2.179,4.541)--(2.308,4.567)--(2.436,4.589)--(2.564,4.610)--(2.692,4.629)--(2.821,4.645)--(2.949,4.661)--(3.077,4.675)--(3.205,4.688)--(3.333,4.700)--(3.462,4.711)--(3.590,4.721)--(3.718,4.731)--(3.846,4.740)--(3.974,4.748)--(4.103,4.756)--(4.231,4.764)--(4.359,4.771)--(4.487,4.777)--(4.615,4.783)--(4.744,4.789)--(4.872,4.795)--(5.000,4.800);
\node[anchor=west] at (1.5,2) (l1) {$Q_1\mathbf{J}_{\nu,M}^{\rm ASY}(\, \underline{r}\,\underline{\omega}^\intercal)Q_1$};
\node[anchor=west] at (-.1,4.2) (l2) {\rotatebox{45}{$\mathbf{J}_\nu^{\rm EVAL}(\, \underline{r}\,\underline{\omega}^\intercal)$}};
\node[anchor=south] at (5,4.9) (l2) {\rotatebox{90}{$N$}};
\node[anchor=east] at (0.1,0) (l2) {$N$};
\node[anchor=east] at (0.1,4) (l2) {$\lceil \alpha N^{\frac{1}{2}}\rceil $};
\node[anchor=south] at (1,4.9) (l2) {\rotatebox{90}{$\lceil \alpha N^{\frac{1}{2}}\rceil $}};
\end{tikzpicture}
\caption{A partition of the entries of $\mathbf{J}_\nu(\, \underline{r}\,\underline{\omega}^\intercal )$ that motivates the matrix decomposition~\eqref{eq:StableMatrixDecomposition}.
The asymptotic expansion~\eqref{eq:HankelExpansion} is used for the $(k,n)$ entry 
if it satisfies $\min(k,n)\geq \lceil \alpha N^{\frac{1}{2}}\rceil$. The gray region shows the nonzero entries of 
$\mathbf{J}_\nu^{\rm EVAL}(\, \underline{r}\,\underline{\omega}^\intercal )$ that are not dealt with by the asymptotic 
expansion. The black curve indicates 
the entries satisfying $|R_{\nu,M}(r_kw_n)|=\epsilon$. A numerically stable $\mathcal{O}(N^{3/2})$ algorithm for evaluating Schl\"{o}milch expansions follows.  }
\label{fig:decomposition}
\end{figure}
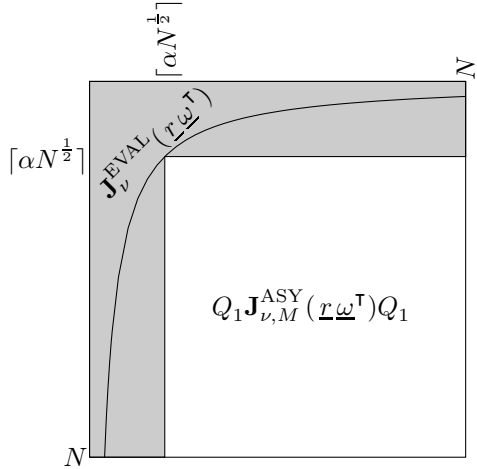

By construction all the entries of $Q_1\mathbf{R}_{\nu,M}(\, \underline{r}\,\underline{\omega}^\intercal )Q_1$ 
are less than $\epsilon$ in magnitude so its 
contribution to the matrix-vector product $\mathbf{J}_\nu(\, \underline{r}\,\underline{\omega}^\intercal )\underline{c}$ 
can be ignored, i.e., we make the approximation
\[
   \mathbf{J}_\nu(\, \underline{r}\,\underline{\omega}^\intercal )\underline{c} \approx Q_1\mathbf{J}_{\nu,M}^{\rm ASY}(\, \underline{r}\,\underline{\omega}^\intercal )Q_1\underline{c} + \mathbf{J}_\nu^{\rm EVAL}(\, \underline{r}\,\underline{\omega}^\intercal )\underline{c}.
\] 
The matrix-vector product $Q_1\mathbf{J}_{\nu,M}^{\rm ASY}(\, \underline{r}\,\underline{\omega}^\intercal )Q_1\underline{c}$
can be computed in $\mathcal{O}(N\log N)$ operations using
DCTs and DSTs by~\eqref{eq:AsymptoticDecomposition}. Moreover, the number of nonzero entries in 
$\mathbf{J}_\nu^{\rm EVAL}(\, \underline{r}\,\underline{\omega}^\intercal )$ is 
$\mathcal{O}(\alpha N^{1/2} N) = \mathcal{O}(N^{3/2})$ and hence, the matrix-vector product 
$\mathbf{J}_\nu^{\rm EVAL}(\, \underline{r}\,\underline{\omega}^\intercal )\underline{c}$ can be computed in $\mathcal{O}(N^{3/2})$ 
operations using direct summation.

This algorithm successfully stabilizes the $\mathcal{O}(N\log N)$ matrix-vector product that 
results from~\eqref{eq:asyR} as it only
employs the asymptotic expansion for entries of $J_\nu(r_k\omega_n)$ for which $r_k\omega_n$ is sufficiently large. In practice, this algorithm is faster than direct summation for $N\geq 100$ (see Figure~\ref{fig:SchlomilchResults}). However, the algorithm 
has an $\mathcal{O}(N^{3/2})$ complexity because the criterion~\eqref{eq:criterion} is  
in hindsight too cautious. We can do better. 

\subsection{A faster evaluation scheme for Schl\"{o}milch expansions}\label{sec:recursivePartition}
Note that the asymptotic expansion~\eqref{eq:HankelExpansion} is accurate 
for all $\mathcal{O}(N^2)$ entries of 
$\mathbf{J}_\nu(\, \underline{r}\,\underline{\omega}^\intercal )$ except for
\[
 \mathcal{O}\left( \sum_{n=1}^N\sum_{k=1}^N \bm{1}_{\{r_k\omega_n< s_{\nu,M}(\epsilon)\}} \right) =  \mathcal{O}\left( \sum_{n=1}^N \alpha^2 N/n \right) = \mathcal{O}\left( N \log N \right)
\]
of them, where $\bm{1}$ is the indicator function and the last equality is from the leading asymptotic behavior of 
the $N$th harmonic number as $N\rightarrow\infty$. Therefore, a lower complexity than 
$\mathcal{O}(N^{3/2})$ seems possible if we use the asymptotic expansion 
for more entries of $\mathbf{J}_\nu(\, \underline{r}\,\underline{\omega}^\intercal )$.
Roughly speaking, we need to refine the 
partitioning of $\mathbf{J}_\nu(\, \underline{r}\,\underline{\omega}^\intercal )$ so 
that the black curve, $|R_{\nu,M}(r_k\omega_n)|=\epsilon$, in Figure~\ref{fig:decomposition}
is better approximated. 

First, we are only allowed to partition $\mathbf{J}_\nu(\, \underline{r}\,\underline{\omega}^\intercal )$ with rectangular matrices 
since for those entries that we do employ the asymptotic expansion we 
need to be able to use DCTs and DSTs to compute matrix-vector products. 
Second, there is a 
balance to be found since each new rectangular partition requires $2M$ more DCTs and DSTs. Too many 
partitions and the cost of computing DCTs and DSTs will be overwhelming, but 
too few and the cost of computing $\mathbf{J}_\nu^{\rm EVAL}(\, \underline{r}\,\underline{\omega}^\intercal )\underline{c}$ will 
dominate. These two competing costs must be balanced.   

To find the balance we introduce a parameter $0<\beta <1$. Roughly speaking, if $\beta\approx 1$ then 
we partition as much 
as possible and the asymptotic expansion is used for every 
entry satisfying $r_k\omega_n\geq s_{\nu,M}(\epsilon)$. If $\beta\approx 0$ then we do 
not refine and we keep the $\mathcal{O}(N^{3/2})$ algorithm from Section~\ref{sec:singlePartition}.
An intermediate value of $\beta$ will balance the two competing costs.

Figure~\ref{fig:Hdecomposition} shows the partition 
of $\mathbf{J}_\nu(\, \underline{r}\,\underline{\omega}^\intercal )$ that we consider and 
it is worth carefully examining that diagram.
The partition corresponds to the following matrix decomposition: 
\begin{equation} 
\begin{aligned} 
\mathbf{J}_\nu(\, \underline{r}\,\underline{\omega}^\intercal ) \approx Q_1\mathbf{J}_{\nu,M}^{\rm ASY}(\, \underline{r}\,\underline{\omega}^\intercal )Q_1 + \sum_{p=1}^P\! \bigg( Q_{2p+1}\mathbf{J}_{\nu,M}^{\rm ASY}(\, \underline{r}\,\underline{\omega}^\intercal )Q_{2p} + Q_{2p}&\mathbf{J}_{\nu,M}^{\rm ASY}(\, \underline{r}\,\underline{\omega}^\intercal )Q_{2p+1}\bigg) \\ 
&+ \mathbf{J}_\nu^{\rm EVAL}(\, \underline{r}\,\underline{\omega}^\intercal ),
\end{aligned}
\label{eq:recursivePartition}
\end{equation} 
where $P$ is the number of partitions, $Q_1$ is the diagonal matrix in Section~\ref{sec:singlePartition}, and 
the $N\times N$ matrices $Q_{2p}$ and $Q_{2p+1}$ are diagonal matrices with 
\[
 (Q_{2p})_{ii} = \begin{cases}1, & \lceil \alpha\beta^{-p}N^{\frac{1}{2}}\rceil \leq i\leq N,\\ 0, & \text{otherwise}, \end{cases} \, (Q_{2p+1})_{ii} = \begin{cases}1, & \lceil\alpha\beta^{p} N^{\frac{1}{2}}\rceil \leq i\leq \lfloor \alpha\beta^{p+1}N^{\frac{1}{2}}\!\rfloor ,\\ 0, & \text{otherwise}. \end{cases}
\]
 
Note that if the $(k,n)$ entry of the matrix $Q_{2p+1}\mathbf{J}_{\nu,M}^{\rm ASY}(\, \underline{r}\,\underline{\omega}^\intercal )Q_{2p}$
is nonzero, then $k\geq \lceil\alpha\beta^{p} N^{\frac{1}{2}}\rceil$ and $n\geq \lceil\alpha\beta^{-p}N^{\frac{1}{2}}\rceil$. Hence, 
$nk \geq \alpha^2 N$ and $r_k\omega_n \geq s_{\nu,M}(\epsilon)$. Similarly, if the $(k,n)$ entry of 
$Q_{2p}\mathbf{J}_{\nu,M}^{\rm ASY}(\, \underline{r}\,\underline{\omega}^\intercal )Q_{2p+1}$ is nonzero, then $r_k\omega_n\geq s_{\nu,M}(\epsilon)$. 
Therefore, we are only employing the asymptotic expansion~\eqref{eq:HankelExpansion}
on entries of $\mathbf{J}_{\nu}(\, \underline{r}\,\underline{\omega}^\intercal)$ for which it is accurate and numerically stable to do so. 

\begin{figure} 
 \centering
 \begin{tikzpicture}
 \draw[fill=black!20] (0,0) -- (0.23255,0) -- (0.23255,.7) -- (0.43478,.7) -- (0.43478,2.7) -- (1,2.7) -- (1,4) -- (2.3,4) -- (2.3,4.56522) -- (4.3,4.56522) -- (4.3,4.76744) -- (5,4.76744) -- (5,5) -- (0,5) -- (0,0);
  \draw (0,0) -- (5,0) -- (5,5) -- (0,5) -- (0,0);
\draw (0.200,0.000)--(0.216,0.366)--(0.221,0.476)--(0.226,0.581)--(0.232,0.682)--(0.237,0.778)--(0.242,0.870)--(0.247,0.957)--(0.253,1.042)--(0.258,1.122)--(0.263,1.200)--(0.268,1.275)--(0.274,1.346)--(0.279,1.415)--(0.284,1.481)--(0.289,1.545)--(0.295,1.607)--(0.300,1.667) --(0.385,2.400)--(0.513,3.050)--(0.641,3.440)--(0.769,3.700)--(0.897,3.886)--(1.026,4.025)--(1.154,4.133)--(1.282,4.220)--(1.410,4.291)--(1.538,4.350)--(1.667,4.400)--(1.795,4.443)--(1.923,4.480)--(2.051,4.513)--(2.179,4.541)--(2.308,4.567)--(2.436,4.589)--(2.564,4.610)--(2.692,4.629)--(2.821,4.645)--(2.949,4.661)--(3.077,4.675)--(3.205,4.688)--(3.333,4.700)--(3.462,4.711)--(3.590,4.721)--(3.718,4.731)--(3.846,4.740)--(3.974,4.748)--(4.103,4.756)--(4.231,4.764)--(4.359,4.771)--(4.487,4.777)--(4.615,4.783)--(4.744,4.789)--(4.872,4.795)--(5.000,4.800);
\node[anchor=west] at (1.5,2) (l1) {$Q_1\mathbf{J}_{\nu,M}^{\rm ASY}(\, \underline{r}\,\underline{\omega}^\intercal)Q_1$};
\node[anchor=west] at (2.4,4.25) (l11) {\footnotesize{$Q_3\mathbf{J}_{\nu,M}^{\rm ASY}(\, \underline{r}\,\underline{\omega}^\intercal)Q_2$}};
\node[anchor=west] at (.42,1.34) (l11) {\footnotesize{\rotatebox{90}{$Q_2\mathbf{J}_{\nu,M}^{\rm ASY}(\, \underline{r}\,\underline{\omega}^\intercal)Q_3$}}};
\node[anchor=west] at (-.1,4.2) (l2) {\rotatebox{45}{$\mathbf{J}_\nu^{\rm EVAL}(\, \underline{r}\,\underline{\omega}^\intercal)$}};
\node[anchor=east] at (0.1,4.5) (l2) {$\lceil\alpha\beta N^{\frac{1}{2}}\rceil$};
\node[anchor=east] at (0.1,4) (l2) {$\lceil\alpha N^{\frac{1}{2}}\rceil$};
\node[anchor=east] at (0.1,2.7) (l2) {$\lceil\alpha\beta^{-1}N^{\frac{1}{2}}\rceil$};
\node[anchor=east] at (0.1,.7) (l2) {$\lceil\alpha\beta^{-2}N^{\frac{1}{2}}\rceil$};
\node[anchor=east] at (0.1,0) (l2) {$N$};
\node[anchor=south] at (5,4.9) (l2) {\rotatebox{90}{$N$}};
\node[anchor=south] at (.9,4.9) (l2) {\rotatebox{90}{$\lceil\alpha N^{\frac{1}{2}}\rceil$}};
\node[anchor=south] at (.35,4.9) (l2) {\rotatebox{90}{$\lceil\beta(\alpha N^{\frac{1}{2}})\rceil$}};
\node[anchor=south] at (5-2.7,4.9) (l2) {\rotatebox{90}{$\lceil\alpha\beta^{-1}N^{\frac{1}{2}}\rceil$}};
\node[anchor=south] at (4.3,4.9) (l2) {\rotatebox{90}{$\lceil\alpha\beta^{-2}N^{\frac{1}{2}}\rceil$}};
\draw[dashed] (0.43478,0)--(0.43478,1);
\draw[dashed] (1,0)--(1,3);
\draw[dashed] (4,4.56522)--(5,4.56522);
\draw[dashed] (2,4)--(5,4);
\end{tikzpicture}
\caption{A partition of the entries of $\mathbf{J}_\nu(\, \underline{r}\,\underline{\omega}^\intercal )$ that motivates the matrix decomposition~\eqref{eq:recursivePartition}.
The gray region shows the nonzero entries of 
$\mathbf{J}_\nu^{\rm EVAL}(\, \underline{r}\,\underline{\omega}^\intercal )$ that are not dealt with by the asymptotic 
expansion. The black curve indicates 
the entries satisfying $|R_{\nu,M}(r_kw_n)|=\epsilon$. 
The asymptotically near-optimal choice of the
parameter $\beta$ is $\mathcal{O}(1/\log N)$ and the number of 
partitions should grow like $\mathcal{O}(\log N/\log\!\log N)$ with $N$.
A numerically stable $\mathcal{O}(N(\log N)^2/\log\!\log N)$ 
algorithm for evaluating Schl\"{o}milch expansions follows. 
}
\label{fig:Hdecomposition}
\end{figure}
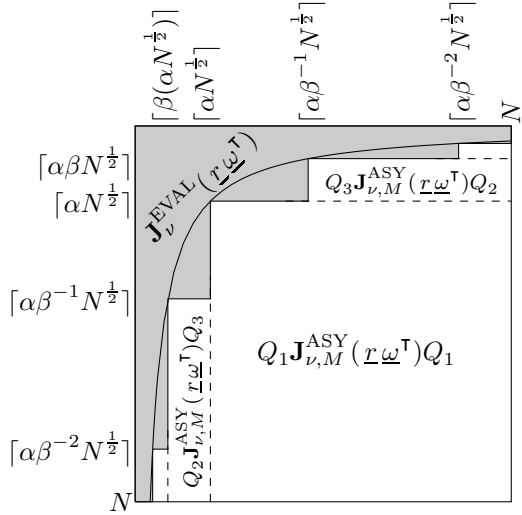

Each matrix-vector product of the form $Q_{2p+1}\mathbf{J}_\nu^{\rm ASY}(\, \underline{r}\,\underline{\omega}^\intercal )Q_{2p}\underline{c}$ 
requires $2M$ DCTs and DSTs from~\eqref{eq:AsymptoticDecomposition} 
and hence, costs $\mathcal{O}(N\log N)$ operations. 
There are a total of $2P+1$ matrices of this form in~\eqref{eq:recursivePartition}
so exploiting the asymptotic expansion requires $\mathcal{O}( P N\log N )$ operations. 

The cost of the matrix-vector product $\mathbf{J}_\nu^{\rm EVAL}(\, \underline{r}\,\underline{\omega}^\intercal )\underline{c}$ 
is proportional to the number of nonzero entries in $\mathbf{J}_\nu^{\rm EVAL}(\, \underline{r}\,\underline{\omega}^\intercal )$. 
Ignoring the constant $\alpha$,  
this is approximately (by counting the entries in rectangular submatrices of $\mathbf{J}_\nu^{\rm EVAL}(\, \underline{r}\,\underline{\omega}^\intercal )$, see Figure~\ref{fig:Hdecomposition})
\[ 
 2\beta^P N^{\frac{3}{2}} + 2\sum_{p=1}^P \beta^{-p} N^{\frac{1}{2}}(\beta^{p-1} N^{\frac{1}{2}} - \beta^{p}N^{\frac{1}{2}} ) = 2\beta^PN^{\frac{3}{2}} + 2P N\left(\frac{1}{\beta} - 1\right) = \mathcal{O}\left( PN/\beta\right),
\]
where the last equality follows from the assumption that $\beta^PN^{1/2} = \mathcal{O}(1)$.

To balance the competing $\mathcal{O}(PN\log N)$ cost of exploiting the asymptotic expansion 
with the $\mathcal{O}(PN/\beta)$ cost of computing $\mathbf{J}_\nu^{\rm EVAL}(\, \underline{r}\,\underline{\omega}^\intercal )\underline{c}$,
we set $PN\log N = PN/\beta$ and find that $\beta = \mathcal{O}(1/\log N)$. Moreover, to
ensure that the assumption $\beta^{P}N^{\frac{1}{2}}=\mathcal{O}(1)$ holds we take
$P = \mathcal{O}( \log N / \log \beta ) = \mathcal{O}( \log N / \log \log N )$.
Thus, the number of partitions should slowly grow with $N$ to balance the competing 
costs. 
With these asymptotically optimal choices of $\beta$ and $P$ the algorithm for 
computing the matrix-vector product $\mathbf{J}_\nu(\, \underline{r}\,\underline{\omega}^\intercal )\underline{c}$
via~\eqref{eq:recursivePartition} requires 
$\mathcal{O}(PN\log N + PN/\beta) = \mathcal{O}( N(\log N)^2/\log \log N)$ operations.
More specifically, the complexity of the algorithm is 
$\mathcal{O}( N(\log N)^2\log(1/\epsilon)/\log \log N)$ since $M=\mathcal{O}(\log(1/\epsilon))$, see Figure~\ref{fig:SchlomilchResults}.

Though, the implicit constants in $\beta = \mathcal{O}(1/\log N)$ 
and $P = \mathcal{O}(\log N/\log\!\log N)$ do not change the complexity of the algorithm, they must still be 
decided on. After numerical experiments we set $\beta = \min(3/\log N,1)$ and 
$P = \lceil \log( 30\alpha^{-1} N^{-1/2} )/\log\beta\rceil$ for computational 
efficiency reasons. 

Table~\ref{tab:algorithmicParameters} summarizes the 
algorithmic parameters that we have carefully selected. The user is 
required to specify the problem with an integer $\nu$, a vector of expansion coefficients 
$\underline{c}$ in~\eqref{eq:Schlomilch}, and then provide a working accuracy 
$\epsilon>0$. All other algorithmic parameters are selected in a 
near-optimal manner based on analysis.

\begin{table} 
\centering
\begin{tabular}{ccc}
\hline
 \rule{0pt}{3ex}Parameter\rule{0pt}{3ex} & Short description & Formula \\[2pt]
 \hline 
 \rule{0pt}{3ex}$M$\rule{0pt}{3ex} & Number of terms in~\eqref{eq:HankelExpansion} & $\max( \lfloor0.3\log(1/\epsilon)\rfloor, 3)$\\[2pt]
 $s_{\nu,M}(\epsilon)$ & $z\geq s_{\nu,M}(\epsilon) \Rightarrow |R_{\nu,M}(z)|\leq \epsilon$ & see Table~\ref{tab:sMTable}\\[2pt]
 $\alpha$ & Partitioning parameter & $(s_{\nu,M}(\epsilon)/\pi)^{1/2}$\\[2pt]
 $\beta$ & Refining parameter & $\min(3/\log N,1)$\\[2pt]
 $P$ & Number of partitions & $\lceil \log( 30\alpha^{-1} N^{-1/2} )/\log\beta\rceil$\\[2pt]
 \hline
\end{tabular}
 \caption{Summary of the algorithmic parameters for computing~\eqref{eq:SchlomilchMatrix}. The user is required to 
 specify the problem with an integer $\nu$, a vector of expansion coefficients 
$\underline{c}$ in~\eqref{eq:Schlomilch}, and provide a working accuracy 
$\epsilon>0$. All other algorithmic parameters 
are carefully selected based on analysis.}
 \label{tab:algorithmicParameters}
\end{table}

\begin{remark}\label{rmk:gammaShift}
A fast evaluation scheme for $f(r) = \sum_{n=1}^N c_n J_\nu((n+\gamma)\pi r)$, where $\gamma$ is a constant, 
immediately follows from this section. 
The constants $d_1$ and $d_2$ in~\eqref{eq:cos} and~\eqref{eq:sin} should be replaced by diagonal 
matrices $D_{\underline{v}}$ and $D_{\underline{w}}$, where $v_k = \cos( -\gamma k\pi/N+(2\nu+1)\pi/4)$ 
and $w_k = \sin( -\gamma k\pi/N+(2\nu+1)\pi/4)$, respectively.
\end{remark}

\subsection{Numerical results for evaluating Schl\"{o}milch expansions}
We now compare three algorithms for evaluating Schl\"{o}milch expansions: (1) 
Direct summation (see Section~\ref{sec:existing}); (2) Our $\mathcal{O}(N^{3/2})$ 
algorithm (see Section~\ref{sec:singlePartition}), and; (3) Our 
$\mathcal{O}(N(\log N)^2/\log\!\log N)$ 
algorithm (see Section~\ref{sec:recursivePartition}). The three algorithms have 
been implemented in MATLAB and are publicly available from~\cite{Townsend_15_01}. 

Figure~\ref{fig:SchlomilchResults} (left) shows the 
execution time for the three algorithms for evaluating Schl\"{o}milch expansions for $\nu=0$. 
Here, we select the working accuracy as $\epsilon = 10^{-15}$ and use 
Table~\ref{tab:algorithmicParameters}
to determine the other algorithmic parameters. We find that our 
$\mathcal{O}(N(\log N)^2/\log\!\log N)$ algorithm in Section~\ref{sec:recursivePartition} 
is faster than direct summation for $N\geq 100$ and faster than our $\mathcal{O}(N^{3/2})$ 
algorithm for $N\geq 160$. In fact, for $N\leq 158$ the number of partitions, $P$,  
is selected to be $0$ and the $\mathcal{O}(N(\log N)^2/\log\!\log N)$ algorithm in Section~\ref{sec:recursivePartition}
is identical to the $\mathcal{O}(N^{3/2})$ algorithm in Section~\ref{sec:singlePartition}.  
For a relaxed working accuracy of $\epsilon = 10^{-3}$ or $\epsilon=10^{-8}$
our $\mathcal{O}(N(\log N)^2/\log\!\log N)$ algorithm becomes even more 
computationally advantageous over direct summation. 

Figure~\ref{fig:SchlomilchResults} (right) shows the execution 
time for our $\mathcal{O}(N(\log N)^2/\log\!\log N)$ algorithm with $N = 5,\!000$, 
$1\leq M\leq 20$, and working accuracies $\epsilon=10^{-15},10^{-11},\ldots,10^{-3}$.  
For each $\epsilon>0$ there is a choice of $M$ that minimizes the
execution time.  Numerical experiments like these motivate
the choice $M = \max( \lfloor0.3\log(1/\epsilon)\rfloor, 3)$ in~\eqref{eq:AsymptoticDecomposition}.

\begin{figure} 
 \centering 
\begin{minipage}{.49\textwidth} 
 \centering 
 \begin{overpic}[width=\textwidth]{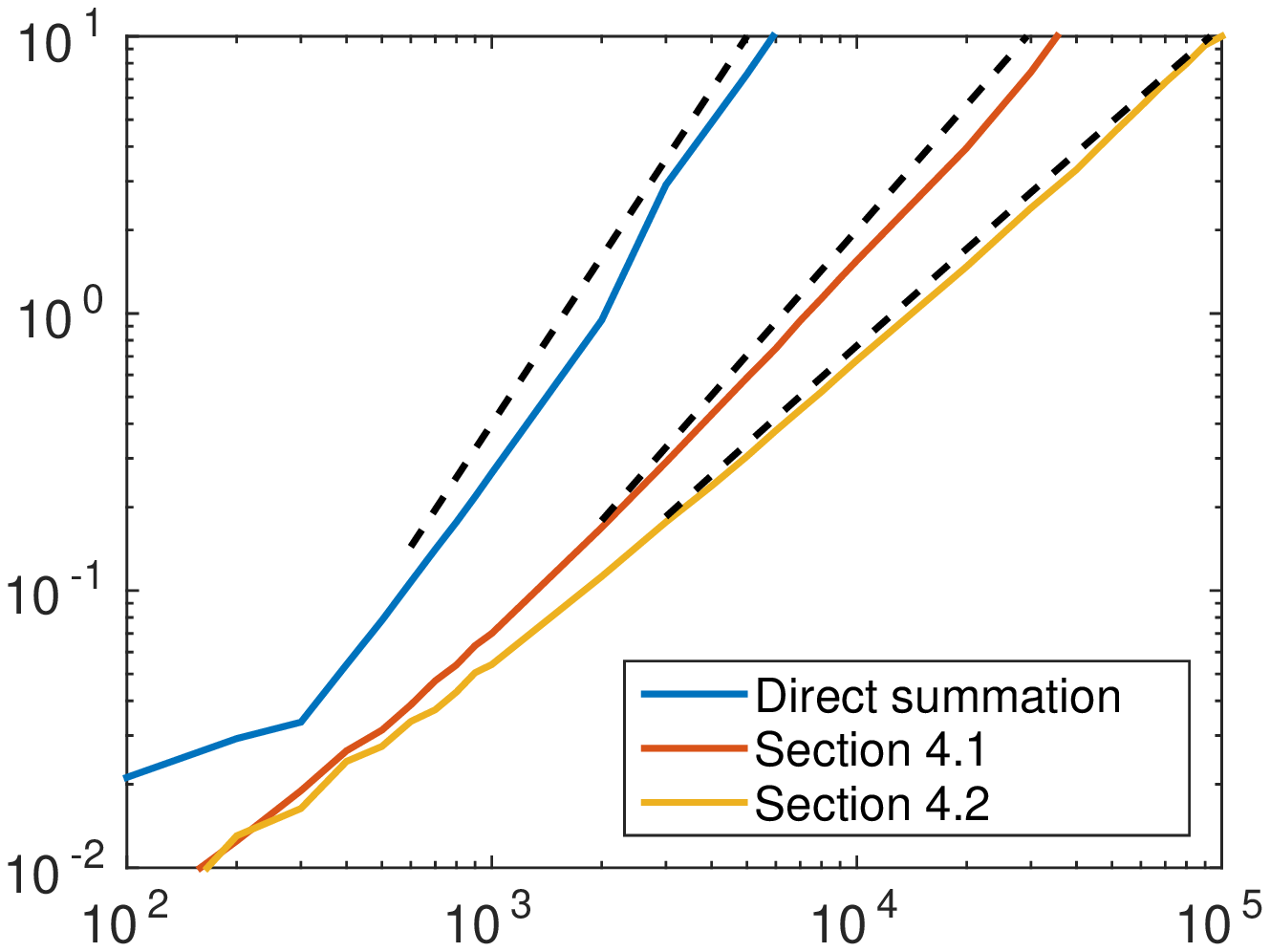}
  \put(41,53) {\footnotesize{\rotatebox{60}{$\mathcal{O}(N^2)$}}}
  \put(53,48) {\footnotesize{\rotatebox{46}{$\mathcal{O}(N^{3/2})$}}}
  \put(51,27) {\footnotesize{\rotatebox{42}{$\mathcal{O}(N(\log N)^2/\log\!\log N)$}}}
  \put(50,0) {$N$}
  \put(0,20) {\rotatebox{90}{Execution time}}
 \end{overpic}
\end{minipage}
\begin{minipage}{.49\textwidth} 
 \centering 
 \begin{overpic}[width=\textwidth]{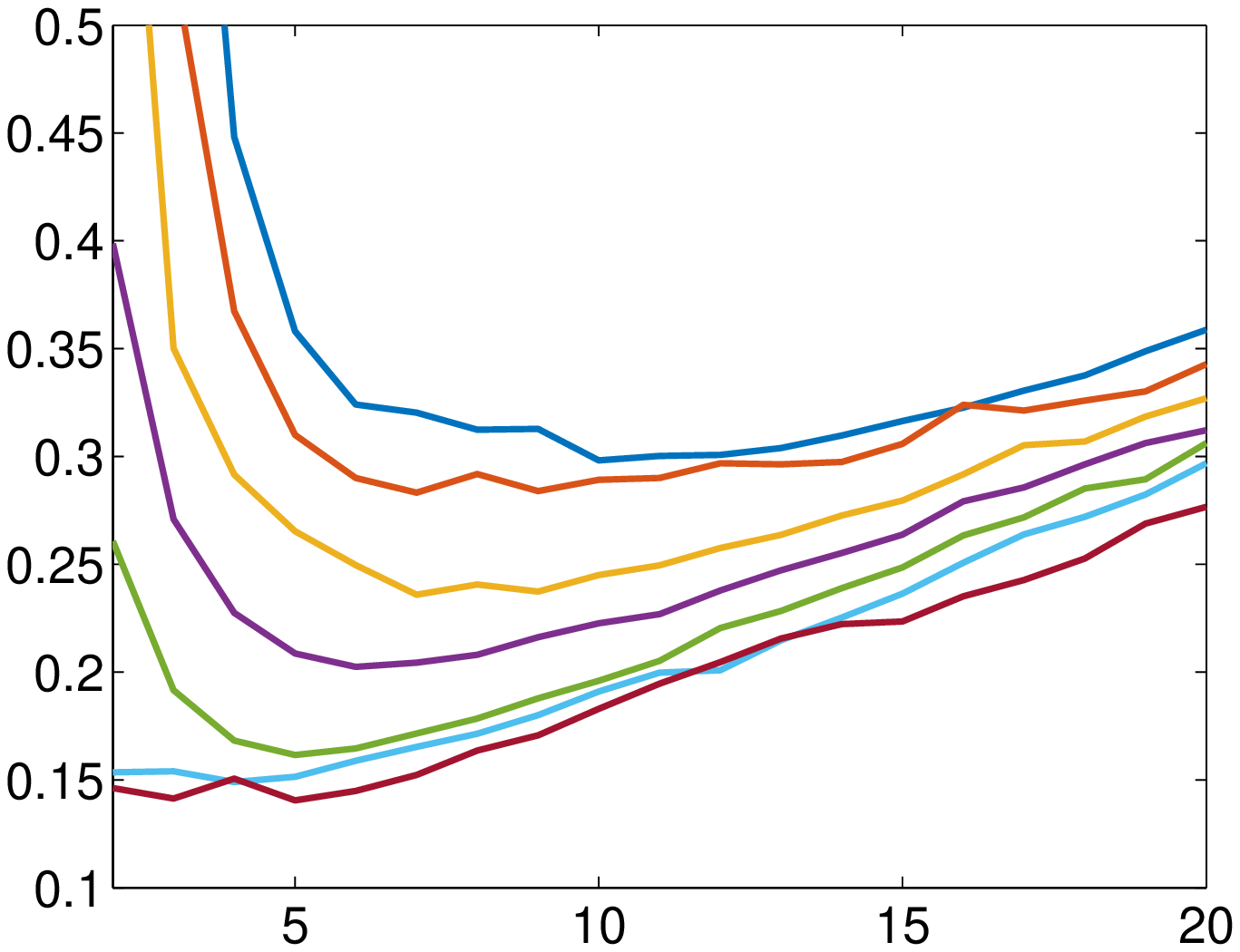}
 \put(50,0) {$M$}
 \put(0,20) {\rotatebox{90}{Execution time}} 
 \put(32,44) {\footnotesize{\rotatebox{-10}{$\epsilon=10^{-15}$}}}
 \put(22,28.2) {\footnotesize{\rotatebox{-16}{$\epsilon=10^{-9}$}}}
 \put(20,11) {\footnotesize{$\epsilon=10^{-3}$}}
 \end{overpic}
\end{minipage}
 \caption{Left: Execution time for computing~\eqref{eq:SchlomilchMatrix} with $\mathcal{O}(N^2)$ direct summation, 
 the $\mathcal{O}(N^{3/2})$ algorithm (see Section~\ref{sec:singlePartition}), and the 
 $\mathcal{O}(N(\log N)^2/\log\!\log N)$ algorithm (see Section~\ref{sec:recursivePartition}). 
 Right: Computational timings for $N = 5,\!000$ with $2\leq M\leq 20$ and $\epsilon=10^{-15},10^{-11},\ldots,10^{-3}$. 
 Numerical experiments like these motivate the choice $M = \max( \lfloor0.3\log(1/\epsilon)\rfloor, 3)$ in~\eqref{eq:AsymptoticDecomposition}.}
 \label{fig:SchlomilchResults}
\end{figure}

\section{A fast evaluation scheme for Fourier--Bessel expansions}\label{sec:FourierBessel}
In this section we adapt the $\mathcal{O}(N(\log N)^2/\log\!\log N)$ algorithm in 
Section~\ref{sec:recursivePartition} to an evaluation scheme for expansions of the 
form: 
\begin{equation}
 f(r) = \sum_{n=1}^N c_n J_\nu(rj_{0,n}), \qquad r\in[0,1],
\label{eq:FourierBessel}
\end{equation}
where $\nu$ is an integer and $j_{0,n}$ is the $n$th positive root of $J_0$.
For $\nu=0$,~\eqref{eq:FourierBessel} 
is the Fourier--Bessel expansion of order $0$ of the function $f:[0,1]\rightarrow \mathbb{C}$.
The algorithm that we describe evaluates~\eqref{eq:FourierBessel} at $r_k = k/N$ 
for $1\leq k\leq N$, which is equivalent to computing the matrix-vector product 
$\mathbf{J}_\nu(\, \underline{r}\,\underline{\smash{j}}_0^\intercal )\underline{c}$, where 
$\underline{\smash{j}}_0$ is a column vector with entries $j_{0,1},\ldots,j_{0,N}$.  

Developing a fast algorithm for computing the matrix-vector product $\mathbf{J}_\nu(\, \underline{r}\,\underline{\smash{j}}_0^\intercal )\underline{c}$ is a challenge 
for FFT-based algorithms because the positive roots of $J_0$ are not equally-spaced. Here, 
we are able to derive a fast evaluation scheme by considering $j_{0,1},\ldots,j_{0,N}$ as a 
perturbed equally-spaced grid (see Section~\ref{sec:BesselRootsPerturbed}). 

First, we derive a matrix decomposition for $\mathbf{J}_\nu(\, \underline{u}\,(\underline{v}+\underline{w})^\intercal )$, where $\underline{u}$, $\underline{v}$, and 
$\underline{w}$ are column vectors.  Afterwards, we will consider $\underline{v} + \underline{w}$ as 
a perturbation of $\underline{v}$. 
\begin{proposition} 
For integers $\nu$ and $N\geq 1$ and for column vectors $\underline{u}$, $\underline{v}$, 
and $\underline{w}$ of length $N$, we have the following matrix decomposition:
\begin{equation}
 \mathbf{J}_\nu(\, \underline{u}\,(\underline{v}+\underline{w})^\intercal ) = \sum_{s=-\infty}^\infty \sum_{t=0}^\infty \frac{(-1)^t2^{-2t-s}}{t!(t+s)!}D_{\underline{u}}^{2t+s} \mathbf{J}_{\nu-s}(\, \underline{u}\,\underline{v}^\intercal )  D_{\underline{w}}^{2t+s}.
\label{eq:BesselMatrixDecompostion}
\end{equation} 
 \label{thm:BesselDecomposition}
\end{proposition}
\begin{proof} 
 We prove~\eqref{eq:BesselMatrixDecompostion} by showing that the 
 $(k,n)$ entry of the matrix on the left-hand side and right-hand side 
 are equal for $1\leq k,n\leq N$.
 
 Pick any $1\leq k,n\leq N$. By the Neumann addition formula~\eqref{eq:NeumannAdditionFormula} 
 we have 
 \begin{equation}
  (\mathbf{J}_\nu(\, \underline{u}\,(\underline{v}+\underline{w})^\intercal ))_{kn} = J_{\nu}( u_kv_n + u_kw_n ) = \sum_{s=-\infty}^\infty J_{\nu-s}( u_kv_n ) J_{s}(u_kw_n).
\label{eq:NeumannAgain}
\end{equation}
Moreover, by applying the Taylor series expansion of Bessel functions~\eqref{eq:TaylorSeriesInfinite} 
we have 
\begin{equation}
 J_{s}(u_kw_n) = \sum_{t=0}^\infty \frac{(-1)^t2^{-2t-s}}{t!(t+s)!}(u_kw_n)^{2t+s}.
\label{eq:Taylor}
\end{equation} 
By substituting~\eqref{eq:Taylor} into~\eqref{eq:NeumannAgain} we obtain
\[
 (\mathbf{J}_\nu(\, \underline{u}\,(\underline{v}+\underline{w})^\intercal ))_{kn} = \sum_{s=-\infty}^\infty\sum_{t=0}^\infty \frac{(-1)^t2^{-2t-s}}{t!(t+s)!} J_{\nu-s}( u_kv_n ) (u_kw_n)^{2t+s},
\]
and the result follows. 
\end{proof}

We now wish to write down a particular instance of~\eqref{eq:BesselMatrixDecompostion} that is 
useful for computing $\mathbf{J}_\nu(\, \underline{r}\,\underline{\smash{j}}_0^\intercal )\underline{c}$.
By Lemma~\ref{lem:BesselInequalities} we can write $j_{0,n} = \tilde{\omega}_n + b_n$, where 
$\tilde{\omega}_n = (n-1/4)\pi$ and $b_n$ is a number such that $0\leq b_n\leq 1/(8(n-1/4)\pi)$. In 
vector notation we have $\underline{\smash{j}}_0 = \tilde{\underline{\omega}} + \underline{b}$. Hence,
by Proposition~\ref{thm:BesselDecomposition} we have 
\begin{equation}
 \mathbf{J}_\nu(\, \underline{r}\,\underline{\smash{j}}_0^\intercal ) = \mathbf{J}_\nu(\, \underline{r}\,(\tilde{\underline{\omega}}+\underline{b})^\intercal ) = \sum_{s=-\infty}^\infty \sum_{t=0}^\infty \frac{(-1)^t2^{-2t-s}}{t!(t+s)!}D_{\underline{r}}^{2t+s} \mathbf{J}_{\nu-s}(\, \underline{r}\,\tilde{\underline{\omega}}^\intercal )  D_{\underline{b}}^{2t+s}.
\label{eq:doubleRank1Sum}
\end{equation}
This is a useful matrix decomposition since each term in the double sum can be 
applied to a vector in $\mathcal{O}( N (\log N)^2/\log\!\log N)$ operations. 
The diagonal matrices $D_{\underline{b}}^{2t+s}$ and $D_{\underline{r}}^{2t+s}$ have 
$\mathcal{O}(N)$ matrix-vector products and $\mathbf{J}_{\nu-s}(\, \underline{r}\,\tilde{\underline{\omega}}^\intercal )$
has $\mathcal{O}( N (\log N)^2/\log\!\log N)$ matrix-vector products (see Section~\ref{sec:recursivePartition} and Remark~\ref{rmk:gammaShift}). 
However, for~\eqref{eq:doubleRank1Sum} to be practical 
we must truncate the inner and outer sums. 

Let $K\geq 1$ be an integer and truncate the outer sum in~\eqref{eq:doubleRank1Sum} to 
$\sum_{s=-K+1}^{K-1}$. By Lemma~\ref{lem:NeumannBound}, with $z=r_k\tilde{\omega}_n$ and $\delta z=r_kb_n$, the 
truncated Neumann addition formula is still accurate, up to an error of $\epsilon$, for the 
$(k,n)$ entry of $\mathbf{J}_\nu(\, \underline{r}\,\underline{\smash{j}}_0^\intercal )$ 
provided that
\[
 |r_kb_n| \leq |b_n| \leq \frac{1}{8(n-\frac{1}{4})\pi} \leq \frac{2}{e}\left(\frac{\epsilon}{5.2}\right)^{\frac{1}{K}}.
\]
Solving for $n$ we find that
\begin{equation}
 n\geq \frac{e}{16\pi}\left(\frac{5.2}{\epsilon}\right)^{\frac{1}{K}} + \frac{1}{4} =\vcentcolon p_{K}(\epsilon).
\label{eq:pKe}
\end{equation}
That is, once we truncate the outer sum in~\eqref{eq:doubleRank1Sum} 
to $\sum_{s=-K+1}^{K-1}$ the matrix decomposition is still accurate
for all the columns of $\mathbf{J}_\nu(\, \underline{r}\,\underline{\smash{j}}_0^\intercal )$ 
except the first $\lfloor p_K(\epsilon) \rfloor$. For example when $K=6$, $p_6(10^{-15})\approx 22.7$ 
and hence, once the outer sum is truncated we must not use~\eqref{eq:doubleRank1Sum} on 
the first $22$ columns of $\mathbf{J}_\nu(\, \underline{r}\,\underline{\smash{j}}_0^\intercal )$. 
Based on numerical experiments we pick $K\geq 1$ to be the smallest integer so that 
$p_K(\epsilon)\leq 30$ and hence, $K = \mathcal{O}(\log(1/\epsilon))$.

Next, we let $T\geq 1$ be an integer and truncate the inner sum to $\sum_{t = 0}^{T-1}$. 
By Section~\ref{sec:BesselTaylor} the truncated Taylor series expansion is 
accurate, up to an error of $\epsilon$, for the $(k,n)$ entry of $\mathbf{J}_\nu(\, \underline{r}\,\underline{\smash{j}}_0^\intercal )$ provided that 
\[
 |r_kb_n| \leq |b_n| \leq \frac{1}{8(n-\frac{1}{4})\pi} \leq \min_{0 \leq s \leq K-1} t_{s,T}(\epsilon)\approx \left(2^{2T}(T!)^2\epsilon\right)^{\frac{1}{2T}},
\]
where the minimization is taken over $0\leq s\leq K-1$, instead of $-K+1\leq s\leq K-1$, because 
of the relation $J_{-s}(z) = (-1)^sJ_s(z)$~\cite[(10.4.1)]{NISTHandbook}.
Solving for $n$ we find that
\begin{equation}
  n\geq \frac{\epsilon^{-1/(2T)}}{16\pi (T!)^{1/T}} + \frac{1}{4} =\vcentcolon q_{T}(\epsilon).
\label{eq:qTe}
\end{equation}
Thus, once the inner sum is truncated in~\eqref{eq:doubleRank1Sum} to $\sum_{t=0}^{T-1}$
the decomposition is accurate for all but the first 
$\lfloor q_T(\epsilon) \rfloor$ columns of 
$\mathbf{J}_\nu(\, \underline{r}\,\underline{\smash{j}}_0^\intercal )$. For example when $T=3$, 
$q_3(10^{-15}) \approx 3.7$ and we must 
not use~\eqref{eq:doubleRank1Sum}
on the first $3$ columns of $\mathbf{J}_\nu(\, \underline{r}\,\underline{\smash{j}}_0^\intercal )$.
In practice, we select $T\geq 1$ to be the smallest integer so that 
$q_T(\epsilon)\leq 30$.

To avoid employing the decomposition~\eqref{eq:doubleRank1Sum} with truncated sums $\sum_{-K+1}^{K-1}\sum_{t=0}^{T-1}$ on the 
first $\lfloor\max(p_K(\epsilon),q_T(\epsilon))\rfloor$ columns of $\mathbf{J}_\nu(\, \underline{r}\,\underline{\smash{j}}_0^\intercal )$, 
we partition the vector $\underline{c}$. That is, we write 
$\mathbf{J}_\nu(\, \underline{r}\,\underline{\smash{j}}_0^\intercal )\underline{c} = \mathbf{J}_\nu(\, \underline{r}\,\underline{\smash{j}}_0^\intercal )\underline{c}_1 + \mathbf{J}_\nu(\, \underline{r}\,\underline{\smash{j}}_0^\intercal )\underline{c}_2$, where
\[
 (\underline{c}_1)_n = \begin{cases} c_n, & n>\lfloor\max(p_K(\epsilon),q_T(\epsilon))\rfloor,\\ 0, & {\rm otherwise}, \end{cases} \qquad (\underline{c}_2)_n = \begin{cases} c_n, & n\leq\lfloor\max(p_K(\epsilon),q_T(\epsilon))\rfloor,\\ 0,& {\rm otherwise}. \end{cases}
\]
First, we compute $\underline{f}_1 =\mathbf{J}_\nu(\, \underline{r}\,\underline{\smash{j}}_0^\intercal )\underline{c}_1$
using~\eqref{eq:doubleRank1Sum} with truncated sums $\sum_{-K+1}^{K-1}\sum_{t=0}^{T-1}$ in $\mathcal{O}(N(\log N)^2/\log\!\log N)$ operations 
and then compute $\underline{f}_2 =\mathbf{J}_\nu(\, \underline{r}\,\underline{\smash{j}}_0^\intercal )\underline{c}_2$ using direct summation 
in $\mathcal{O}(N)$ operations.

\subsection{Rearranging the computation for efficiency}\label{sec:rearrange}
At first it seems that computing $\mathbf{J}_\nu(\, \underline{r}\,\underline{\smash{j}}_0^\intercal )\underline{c}$
is $(2K-1)T$ times the cost of evaluating a Schl\"{o}milch expansion, since there are 
$(2K-1)T$ terms in $\sum_{-K+1}^{K-1}\sum_{t=0}^{T-1}$. However, the computation can be rearranged so the evaluation of~\eqref{eq:FourierBessel}
is only $2T+K-2$ times the cost.  For $K=6$ 
and $T=3$, the values that we take when $\epsilon=10^{-15}$, we have $(2K-1)T=33$ and $2T+K-2=10$ so there is a significant 
computational saving here.

Since $J_{-s}(z) = (-1)^sJ_{s}(z)$~\cite[(10.4.1)]{NISTHandbook}
we have
\[
\begin{aligned}
  \mathbf{J}&_\nu(\, \underline{r}\,\underline{\smash{j}}_0^\intercal )\underline{c} = \sum_{s=-K+1}^{K-1} \sum_{t=0}^{T-1} \frac{(-1)^t2^{-2t-s}}{t!(t+s)!}D_{\underline{r}}^{2t+s} \mathbf{J}_{\nu-s}(\, \underline{r}\,\tilde{\underline{\omega}}^\intercal )  D_{\underline{b}}^{2t+s}\underline{c}\\
  & = \sum_{s=0}^{K-1}{\!\!}^{'} \sum_{t=0}^{T-1} \frac{(-1)^t2^{-2t-s}}{t!(t+s)!}D_{\underline{r}}^{2t+s} \left[ \mathbf{J}_{\nu-s}(\, \underline{r}\,\tilde{\underline{\omega}}^\intercal ) + (-1)^s\mathbf{J}_{\nu+s}(\, \underline{r}\,\tilde{\underline{\omega}}^\intercal ) \right] D_{\underline{b}}^{2t+s}\underline{c}\\
  & = \!\!\!\sum_{u=0}^{2T+K-3}\!\!\!\!\! D_{\underline{r}}^{u}\!\! \left[\!\sum_{t=\max(\lceil \frac{u-K+1}{2}\rceil,0)}^{\min(\lfloor \frac{u}{2}\rfloor,T-1)}{}^{\!\!\!\!\!\!\!\!\!\!\!\!\!\!\!\!\!\!\!\!\!''\,\,\,\,\,\,} \frac{(-1)^t2^{-u}}{t!(u-t)!}\!\! \left[\mathbf{J}_{\nu-u+2t}(\, \underline{r}\,\tilde{\underline{\omega}}^\intercal ) + (-1)^{u-2t}\mathbf{J}_{\nu+u-2t}(\, \underline{r}\,\tilde{\underline{\omega}}^\intercal )\right] \! \right] \!\! D_{\underline{b}}^{u}\underline{c},\\
  \end{aligned} 
\]
where the single prime on the sum indicates that the first term is halved and the double prime 
indicates that the last term is halved. Here, the last equality follows by letting 
$u=2t+s$ and performing a change of variables. 
Now, for each $0\leq u\leq 2T+K-3$, the inner sum can be computed at the cost of
one application of the $\mathcal{O}(N(\log N)^2/\log\!\log N)$ algorithm in Section~\ref{sec:recursivePartition}.  
Since each evaluation of a Schl\"{o}milch expansion costs 
$\mathcal{O}(N(\log N)^2\log(1/\epsilon)/\log\!\log N)$ operations and 
$K=\mathcal{O}(\log(1/\epsilon))$, a Fourier--Bessel evaluation costs $\mathcal{O}(N(\log N)^2\log(1/\epsilon)^2/\log\!\log N)$ 
operations.
\subsection{Numerical results for evaluating Fourier--Bessel expansions}
Figure~\ref{fig:FBResults1} (left) shows the 
execution time for evaluating Fourier--Bessel expansions of order $0$ 
with direct summation and our
$\mathcal{O}(N(\log N)^2/\log\!\log N)$ algorithm with working accuracies
of $\epsilon = 10^{-15}, 10^{-8}, 10^{-3}$. 
Our algorithm is faster than direct 
summation for $N\geq 700$ when $\epsilon = 10^{-15}$. 
Relaxing the working accuracy not only decreases $K$ and $T$, but also speeds 
up the evaluation of each Schl\"{o}milch expansion.
When $\epsilon=10^{-3}$, our 
algorithm is faster than direct summation for $N\geq 100$. 

The algorithmic parameters are selected based on careful analysis so that the evaluation 
scheme is numerically stable and the computed vector $\underline{f}$ 
has a componentwise accuracy of $\mathcal{O}(\epsilon \|\underline{c}\|_1)$, where $\|\cdot\|_1$ is the vector 1-norm. 
In Figure~\ref{fig:FBResults1} (right) we verify this for $\epsilon = 10^{-15}, 10^{-8}, 10^{-3}$ by considering 
the error $\|\underline{f}-\underline{f}^{{\rm quad}}\|_\infty$, where  
$\|\cdot\|_\infty$ is the absolute maximum vector norm and
$\underline{f}^{\rm quad}$ is the vector of values computed in 
quadruple precision using direct summation.
For these experiments we take the entries of $\underline{c}$ as 
realizations of independent 
Gaussians so that $\|\underline{c}\|_1=\mathcal{O}(N^{1/2})$.

\begin{figure} 
 \centering 
\begin{minipage}{.49\textwidth} 
 \begin{overpic}[width=\textwidth]{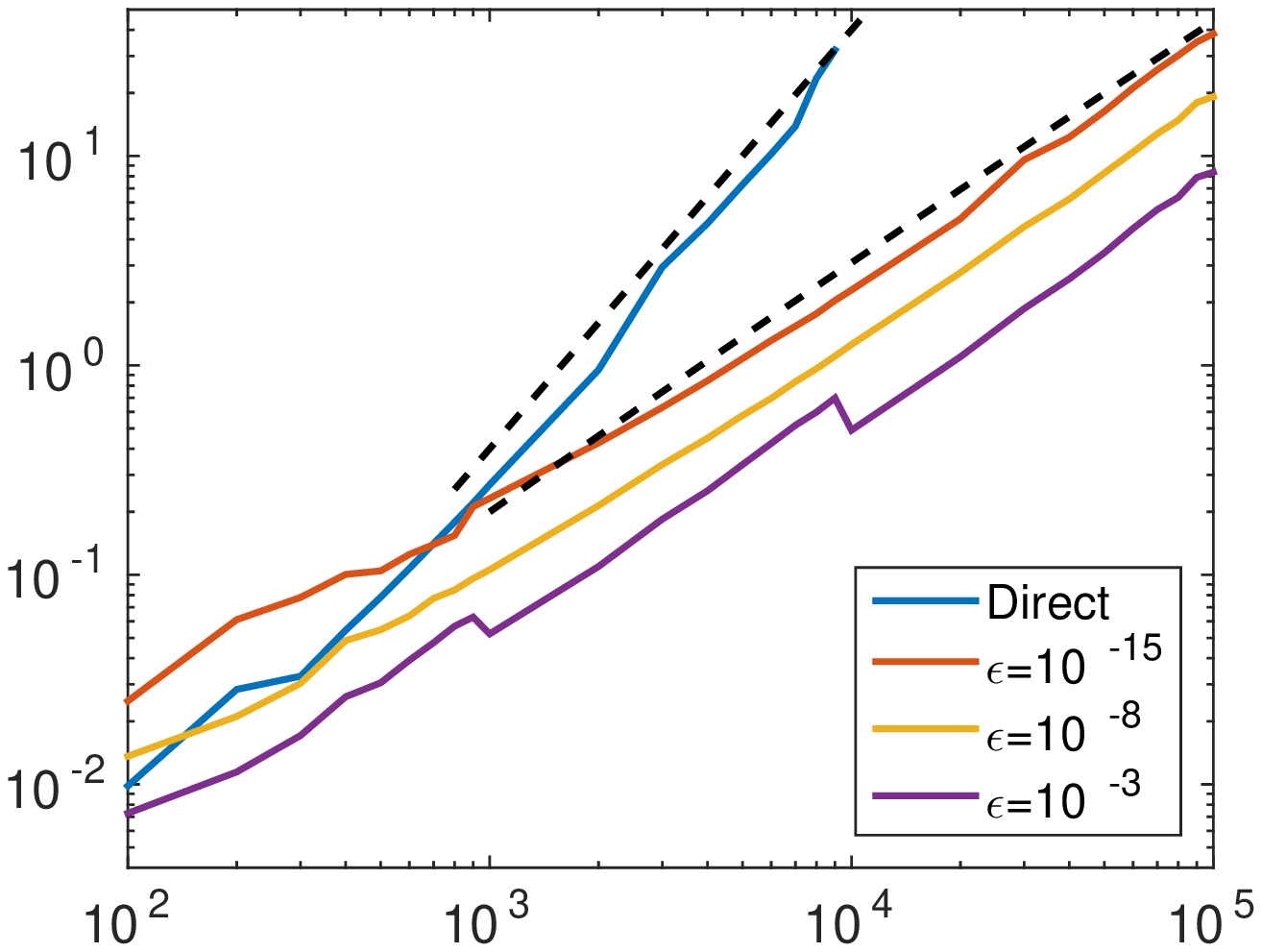}
  \put(47,53) {\footnotesize{\rotatebox{55}{$\mathcal{O}(N^2)$}}}
  \put(43.5,40) {\footnotesize{\rotatebox{34}{$\mathcal{O}(N(\log N)^2/\log\!\log N)$}}}
  \put(50,-1) {$N$}
  \put(-.5,20) {\rotatebox{90}{Execution time}}
 \end{overpic}
\end{minipage}
 \begin{minipage}{.49\textwidth} 
  \begin{overpic}[width=\textwidth]{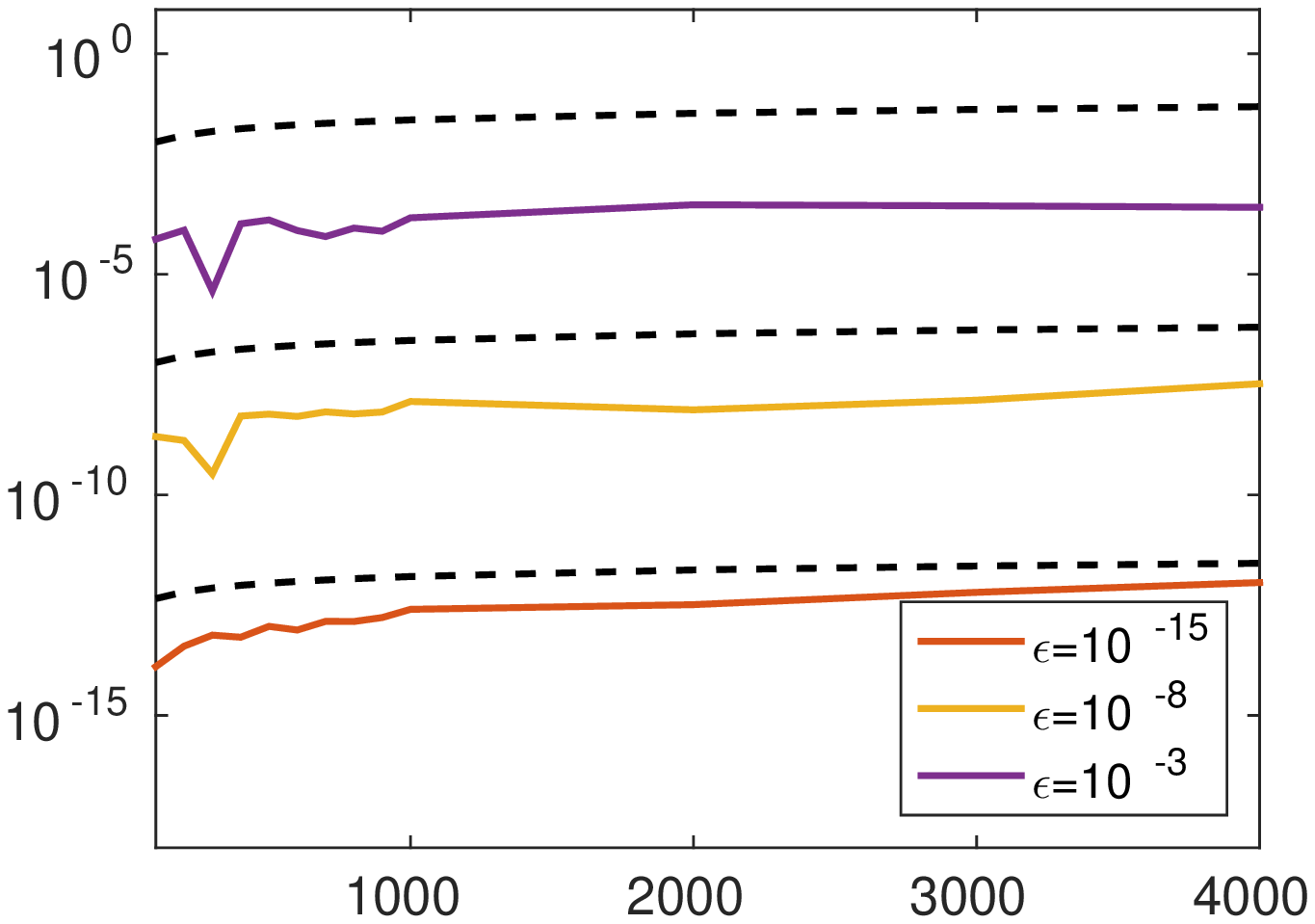}
  \put(45,29.5) {\footnotesize{$10^{-15}\left\|\underline{c}\right\|_1$}}
  \put(45,46.5) {\footnotesize{$10^{-8}\left\|\underline{c}\right\|_1$}}
  \put(45,61.5) {\footnotesize{$10^{-3}\left\|\underline{c}\right\|_1$}}
  \put(50,-1) {$N$}
  \put(-4,23) {\rotatebox{90}{$\|\underline{f} - \underline{f}^{{\rm quad}} \|_\infty$}}
  \end{overpic}
\end{minipage}
 \caption{Left: Execution time for evaluating~\eqref{eq:FourierBessel} at $1/N,\ldots,N/N$ with 
 $\mathcal{O}(N^2)$ direct summation and the $\mathcal{O}(N(\log N)^2/\log\!\log N)$ 
 algorithm described in this section. 
 For a relaxed working accuracy our algorithm becomes even more 
 computationally advantageous over direct summation. Right: 
 Maximum absolute error in the computed vector $\underline{f}$. Here, we 
 compare the observed error (solid lines) with the expected bound of $\mathcal{O}(\epsilon\|\underline{c}\|_1)$ (dotted lines). }
 \label{fig:FBResults1}
\end{figure}

\section{The discrete Hankel transform} \label{sec:DiscreteHankelTransform}
We now adapt the algorithm in Section~\ref{sec:FourierBessel} to the computation 
of the discrete Hankel transform (DHT) of order $0$, which is defined 
as~\cite{Johnson_87_01}
\begin{equation}
 f_k = \sum_{n=1}^N c_n J_0(j_{0,k}j_{0,n}/j_{0,N+1}), \qquad 1\leq k\leq N,
\label{eq:DHT} 
\end{equation} 
where $j_{0,n}$ is the $n$th positive root of $J_0(z)$. It 
is equivalent to the matrix-vector product 
$\mathbf{J}_0(\, \underline{\smash{j}}_0\,\underline{\smash{j}}_0^\intercal/j_{0,N+1})\underline{c}$ 
and is more difficult for FFT-based algorithms than the 
task in Section~\ref{sec:FourierBessel} because it may be regarded as 
evaluating a Fourier--Bessel expansion at points that are not equally-spaced.

To extend our $\mathcal{O}(N(\log N)^2/\log\!\log N)$ algorithm to this setting
we consider the ratios $j_{0,1}/j_{0,N+1},\ldots,j_{0,N}/j_{0,N+1}$ as a perturbed equally-spaced 
grid. By Lemma~\ref{lem:BesselInequalities} we have
\[
 \frac{j_{0,k}}{j_{0,N+1}} = \frac{k-\frac{1}{4}}{N+\frac{3}{4}} + e_{k,N}, \quad |e_{k,N}|\leq \frac{1}{8(N+\frac{3}{4})(k-\frac{1}{4})\pi^2}, \qquad 1\leq k\leq N.
\]
We can write this in vector notation as 
$\underline{\smash{j}}_0/j_{0,N+1} = \tilde{\underline{r}} + \underline{e}_N$, where $\tilde{r}_k = (k-1/4)/(N+3/4)$. 
Since $\mathbf{J}_0(\, (\underline{v}+\underline{w})\underline{u}^\intercal) = \mathbf{J}_0(\, \underline{u}(\underline{v}+\underline{w})^\intercal)^\intercal$ 
and $\mathbf{J}_0(\, \underline{\smash{j}}_0\,\underline{\smash{j}}_0^\intercal/j_{0,N+1}) = \mathbf{J}_0((\tilde{\underline{r}}+\underline{e}_N)\,\underline{\smash{j}}_0^\intercal)$,
we have by Theorem~\ref{thm:BesselDecomposition} the following matrix decomposition:
\begin{equation} 
 \mathbf{J}_0(\, \underline{\smash{j}}_0\,\underline{\smash{j}}_0^\intercal/j_{0,N+1}) = \sum_{s=-\infty}^\infty \sum_{t=0}^\infty \frac{(-1)^t2^{-2t-s}}{t!(t+s)!}D_{\underline{e}_N}^{2t+s} \mathbf{J}_{-s}(\, \tilde{\underline{r}}\,\underline{\smash{j}}_0^\intercal )  D_{\underline{\smash{j}}_0}^{2t+s}.
\label{eq:HankelSum} 
\end{equation} 

Each term in~\eqref{eq:HankelSum} can be applied to a vector 
in $\mathcal{O}( N (\log N)^2/\log\!\log N)$ operations. 
The diagonal matrices $D_{\underline{\smash{j}}_0}^{2t+s}$ and $D_{\underline{e}_N}^{2t+s}$ have 
$\mathcal{O}(N)$ matrix-vector products and by Section~\ref{sec:FourierBessel} 
the matrix-vector product 
$\mathbf{J}_{-s}(\, \tilde{\underline{r}}\,\underline{\smash{j}}_0^\intercal )\underline{c}$ 
can be computed in $\mathcal{O}( N (\log N)^2/\log\!\log N)$ operations.\footnote{Here, $\tilde{r}_k = (k-1/4)/(N+3/4) = (4k-1)/(4N+3)$ so to compute $\mathbf{J}_{-s}(\, \tilde{\underline{r}}\,\underline{\smash{j}}_0^\intercal )\underline{c}$ pad the vector $\underline{c}$ with $3N+3$ zeros, use the algorithm in Section~\ref{sec:FourierBessel} with $N$ replaced by $4N+3$, and then only keep $f_{4i-1}$ for $1\leq i\leq N$.} 
However, for~\eqref{eq:HankelSum} to be practical we must truncate the double sum. 

Let $K\geq 1$ and $T\geq 1$ be integers and consider truncating the sum in~\eqref{eq:HankelSum}
to $\sum_{s=-K+1}^{K-1} \sum_{t=0}^{T-1}$. Using similar reasoning to that in 
Section~\ref{sec:FourierBessel}, if we truncate~\eqref{eq:HankelSum} then the matrix decomposition is accurate, up to an error of $\epsilon$, 
provided that 
\[
 |e_{k,N} j_{0,n}| \leq \frac{j_{0,n}}{8(N+\frac{3}{4})(k-\frac{1}{4})\pi^2}\leq \frac{1 + \frac{1}{8(n-\frac{1}{4})(N+\frac{3}{4})\pi^2}}{8(k-\frac{1}{4})\pi}\leq \frac{1.01}{8(k-\frac{1}{4})\pi} \leq \frac{2}{e}\left(\frac{\epsilon}{5.2}\right)^{\frac{1}{K}}
\]
and 
\[
 |e_{k,N} j_{0,n}|\leq \frac{1.01}{8(k-\frac{1}{4})\pi}\leq \min_{0 \leq s \leq K-1} t_{s,T}(\epsilon)\approx \left(2^{2T}(T!)^2\epsilon\right)^{\frac{1}{2T}}.
\]
where we used the bound $j_{0,n}/((N+3/4)\pi)\leq (n-1/4)/(N+3/4) + 1/(8(n-1/4)(N+3/4)\pi^2)\leq 1.01$ for $n,N\geq 1$.
Equivalently, provided that  
\[
 k\geq 1.01\times\max( p_K(\epsilon), q_T(\epsilon) ), 
\]
where $p_K(\epsilon)$ and $q_T(\epsilon)$ are defined in~\eqref{eq:pKe} and~\eqref{eq:qTe}, respectively. 
That is, after truncating the sums in~\eqref{eq:HankelSum} we can use the decomposition 
on all the entries of $\mathbf{J}_0(\, \underline{\smash{j}}_0\,\underline{\smash{j}}_0^\intercal/j_{0,N+1})$ 
except for those in the first $\lfloor 1.01\max( p_K(\epsilon), q_T(\epsilon) )\rfloor$ rows. 
In practice, we take $K\geq 1$ and $T\geq 1$ to be the smallest integers so 
that $1.01\max( p_K(\epsilon), q_T(\epsilon) )\leq 30$.

To avoid employing a truncated version of~\eqref{eq:HankelSum} on the first few 
rows we compute $\underline{f} = \mathbf{J}_0(\, \underline{\smash{j}}_0\,\underline{\smash{j}}_0^\intercal/j_{0,N+1})\underline{c}$
using~\eqref{eq:HankelSum} with sums $\sum_{s=-K+1}^{K-1}\sum_{t=0}^{T-1}$ in $\mathcal{O}(N(\log N)^2/\log\!\log N)$ operations, discard
the first $\lfloor 1.01\max( p_K(\epsilon), q_T(\epsilon) )\rfloor$ entries of $\underline{f}$,
and then compute the discarded entries using direct summation in $\mathcal{O}(N)$ operations. 

In the same way as Section~\ref{sec:rearrange}, we can rearrange the computation so 
that only $2T+K-2$, instead of $(2K-1)T$, Fourier--Bessel evaluations are 
required.  Each Fourier--Bessel evaluation requires $2T+K-2$ Schl\"{o}milch 
evaluations. Hence, a discrete Hankel transform requires $(2T+K-2)^2$
Schl\"{o}milch evaluations. Since $K=\log(1/\epsilon)$, our algorithm for the DHT 
requires $\mathcal{O}(N(\log N)^2\log(1/\epsilon)^3/\log\!\log N)$ operations.

The exposition in this section may suggest that the same algorithmic ideas 
continue to work for $\nu>0$. While this is correct, the computational performance rapidly 
deteriorates as $\nu$ increases. This is because the vector
of Bessel roots $\underline{\smash{j}}_\nu$ and the vector of ratios 
$\underline{\smash{j}}_\nu^\intercal/j_{\nu,N+1}$ are distributed less like an 
equally-spaced grid for $\nu>0$. This significantly increases the number 
of terms required in the Taylor expansion and Neumann addition formula. 
For large $\nu>0$, different algorithmic ideas should be employed, for example,
one should take into account that $|J_\nu(z)|$ can be very 
small when $0\leq z < \nu$. It is expected that the matched asymptotic 
(transition) region of $J_\nu(z)$, when $z\approx \nu$, will be the most 
difficult to exploit when deriving a fast algorithm. 

\subsection{Numerical results for computing the discrete Hankel transform}
We have implemented our $\mathcal{O}(N(\log N)^2/\log\!\log N)$ algorithm 
for computing the DHT in MATLAB. It is publicly available from~\cite{Townsend_15_01}.

Figure~\ref{fig:HankelResults1} (left) shows the 
execution time for computing the DHT of order $0$ 
with direct summation and our
$\mathcal{O}(N(\log N)^2/\log\!\log N)$ algorithm with working tolerances 
of $\epsilon = 10^{-15}, 10^{-8}, 10^{-3}$. 
Our algorithm is faster than direct 
summation for $N\geq 6,\!000$ when $\epsilon = 10^{-15}$, for $N\geq 2,\!000$ 
when $\epsilon = 10^{-8}$, and for $N\geq 100$ when $\epsilon =10^{-3}$. 
Figure~\ref{fig:HankelResults1} (right) verifies that the selected
algorithmic parameters do achieve the desired error bound of 
$\|\underline{f}-\underline{f}^{{\rm quad}}\|_\infty = \mathcal{O}(\epsilon \|\underline{c}\|_1)$, 
where $\underline{f}$ is the vector of the computed values of~\eqref{eq:DHT} 
and $\underline{f}^{\rm quad}$ is the vector of values computed in 
quadruple precision using direct summation. In this experiment we 
take the entries of $\underline{c}$ to be realizations of independent Gaussians (with mean $0$ and variance $1$)
so that $\|\underline{c}\|_1 = \mathcal{O}(N^{1/2})$.
\begin{figure} 
 \centering 
\begin{minipage}{.49\textwidth}
 \begin{overpic}[width=\textwidth]{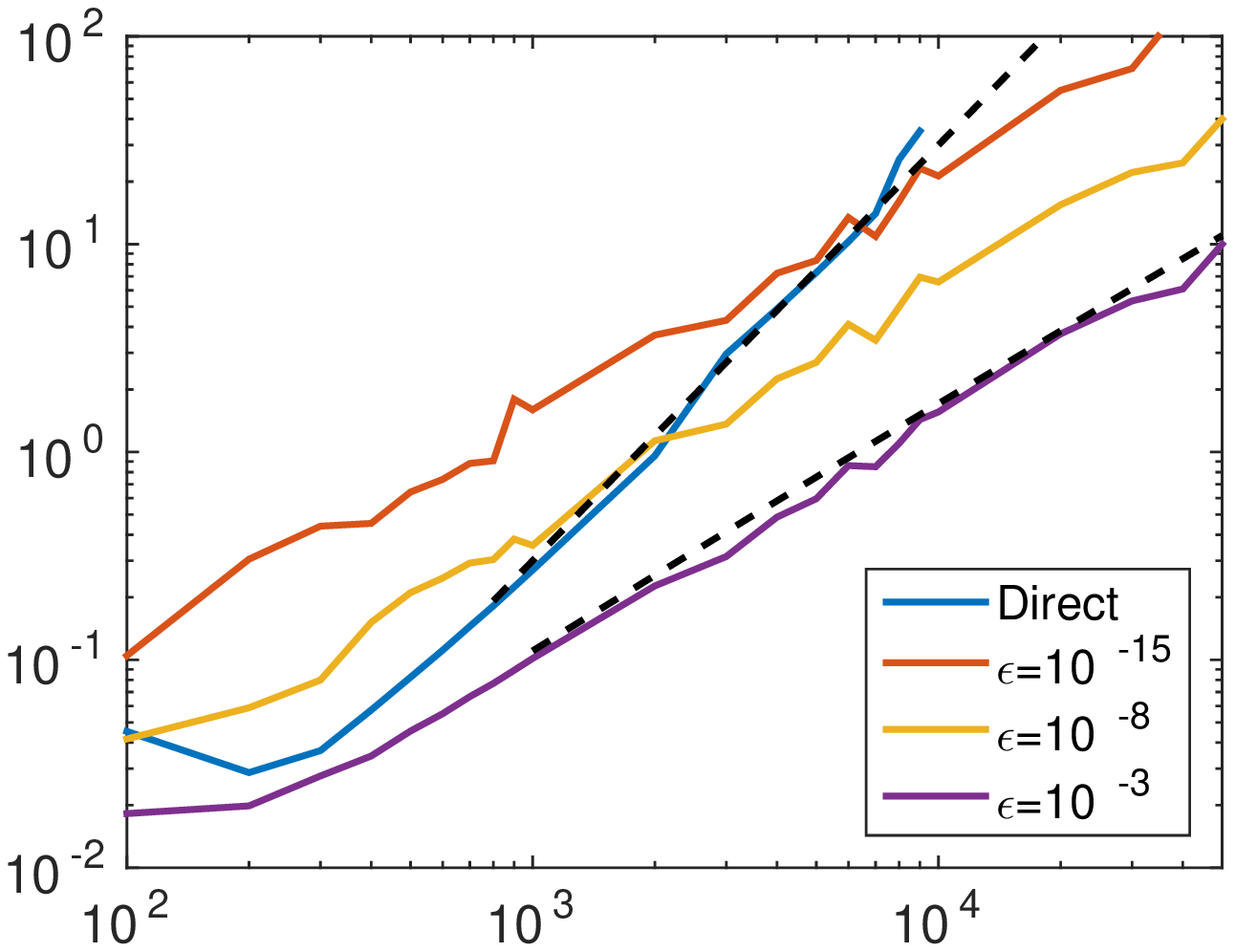}
 \put(60,55) {\footnotesize{\rotatebox{45}{$\mathcal{O}(N^2)$}}}
 \put(46,30) {\footnotesize{\rotatebox{31}{$\mathcal{O}(N(\log N)^2/\log\!\log N)$}}}
 \put(50,-1) {$N$}
 \put(0,18) {\rotatebox{90}{Execution time}}
 \end{overpic}
\end{minipage}
\begin{minipage}{.49\textwidth}
 \begin{overpic}[width=\textwidth]{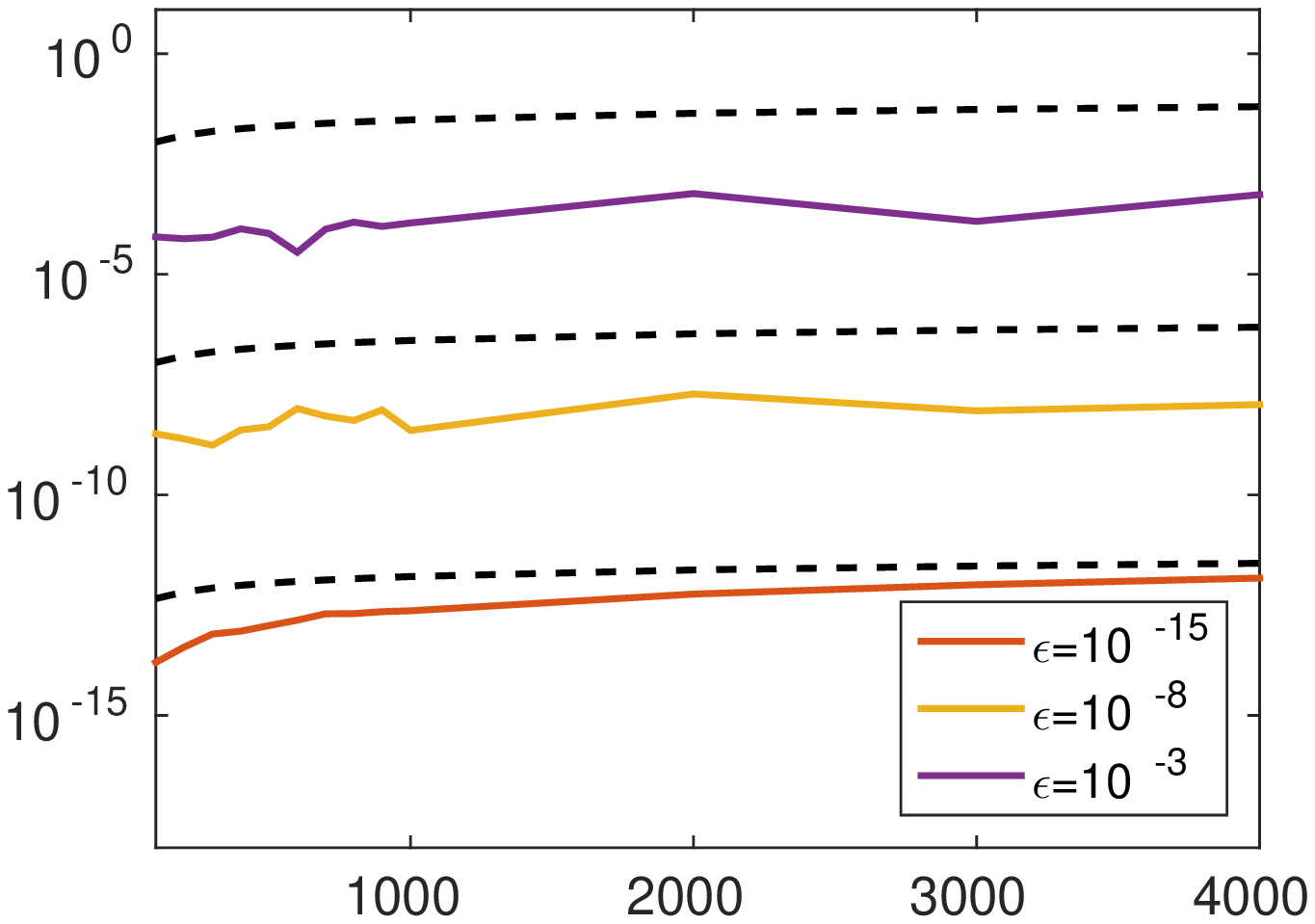}
    \put(45,30) {\footnotesize{$10^{-15}\left\|\underline{c}\right\|_1$}}
  \put(45,47) {\footnotesize{$10^{-8}\left\|\underline{c}\right\|_1$}}
  \put(45,62) {\footnotesize{$10^{-3}\left\|\underline{c}\right\|_1$}}
  \put(50,-1) {$N$}
  \put(-5,22) {\rotatebox{90}{$\|\underline{f} - \underline{f}^{{\rm quad}} \|_\infty$}}
   \end{overpic}
\end{minipage}
 \caption{Left: Execution time for computing the DHT in~\eqref{eq:DHT} using 
 direct summation and our $\mathcal{O}(N(\log N)^2/\log\!\log N)$ 
 algorithm with $\epsilon=10^{-15},10^{-8},10^{-3}$. Right: The maximum absolute error in the computed vector 
 $\underline{f}$. Here, we compare the computed error (solid lines) with the 
 expected error bound $\mathcal{O}(\epsilon \|\underline{c}\|_1)$.}
 \label{fig:HankelResults1}
\end{figure}

The DHT is, up to a diagonal scaling, its own 
inverse as $N\rightarrow\infty$. In particular, by~\cite{Johnson_87_01} we have, as $N\rightarrow\infty$,
\begin{equation}
 \frac{1}{j_{0,N+1}^2}\mathbf{J}_0(\, \underline{\smash{j}}_0\,\underline{\smash{j}}_0^\intercal/j_{0,N+1} ) D_{\underline{v}} \mathbf{J}_0(\, \underline{\smash{j}}_0\,\underline{\smash{j}}_0^\intercal/j_{0,N+1} ) D_{\underline{v}} \sim I_{N},
\label{eq:selfinverse}
\end{equation} 
where $D_{\underline{v}}$ is a diagonal matrix with entries $v_{n} = 2/J_{1}(j_{0,n})^2$, and $I_N$ is the $N\times N$ identity matrix. 
We can use the relation~\eqref{eq:selfinverse} to verify our algorithm for the DHT. Figure~\ref{fig:HankelResults2} 
shows the error $\|\underline{\tilde{c}} - \underline{c}\|_\infty$, where $\underline{c}$ is a column 
vector with entries drawn from independent Gaussians and
\begin{equation}
 \underline{\tilde{c}} =  \frac{1}{j_{0,N+1}^2}\mathbf{J}_0(\, \underline{\smash{j}}_0\,\underline{\smash{j}}_0^\intercal/j_{0,N+1} ) D_{\underline{v}} \mathbf{J}_0(\, \underline{\smash{j}}_0\,\underline{\smash{j}}_0^\intercal/j_{0,N+1} ) D_{\underline{v}}\underline{c}.
\label{eq:inverseTest}
\end{equation} 
For small $N$, $\|\underline{\tilde{c}} - \underline{c}\|_\infty$ will be large because~\eqref{eq:selfinverse}
only holds as $N\rightarrow\infty$. However, even for large $N$, 
we observe that $\|\underline{\tilde{c}} - \underline{c}\|_\infty$ grows like 
$\mathcal{O}(N^{3/2})$. This can be explained by tracking the accumulation of 
numerical errors. 
Since the entries of $\underline{c}$ are drawn from independent Gaussians and $v_n=\mathcal{O}(n)$ we have 
$\|D_{\underline{v}}\underline{c}\|_1 = \mathcal{O}(N^{3/2})$ and hence, 
we expect $\|\mathbf{J}_0(\, \underline{\smash{j}}_0\,\underline{\smash{j}}_0^\intercal/j_{0,N+1} ) D_{\underline{v}}\underline{c}\|_\infty=\mathcal{O}(\epsilon\|D_{\underline{v}}\underline{c}\|_1)=\mathcal{O}(\epsilon N^{3/2})$. 
By the same reasoning we expect $\|\mathbf{J}_0(\, \underline{\smash{j}}_0\,\underline{\smash{j}}_0^\intercal/j_{0,N+1} ) D_{\underline{v}}\mathbf{J}_0(\, \underline{\smash{j}}_0\,\underline{\smash{j}}_0^\intercal/j_{0,N+1} ) D_{\underline{v}}\underline{c}\|_\infty=\mathcal{O}(\epsilon N^{7/2})$. 
Finally, this gets multiplied by $j_{0,N+1}^{-2}=\mathcal{O}(N^{-2})$ in~\eqref{eq:inverseTest} 
so we expect $\|\underline{\tilde{c}} - \underline{c}\|_\infty=\mathcal{O}(N^{3/2})$.
If in practice we desire 
$\|\underline{\tilde{c}} - \underline{c}\|_\infty=\mathcal{O}(1)$, then it is 
sufficient to have $|c_n| = o(n^{-2})$. 
\begin{figure} 
 \centering 
 \begin{overpic}[width=.49\textwidth]{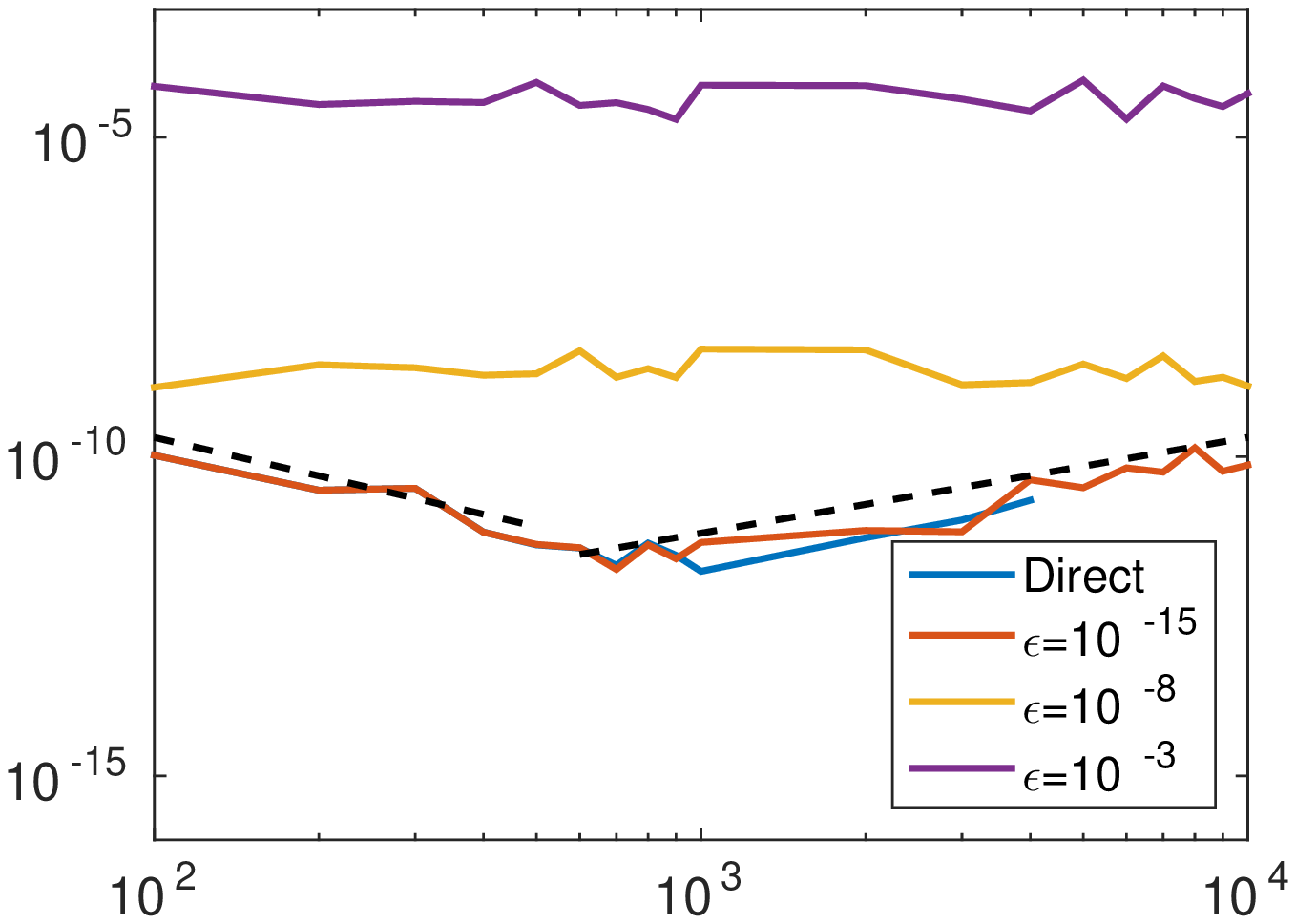}
 \put(52,32) {\footnotesize{\rotatebox{12}{$\mathcal{O}(N^{3/2})$}}}
 \put(23,36) {\footnotesize{\rotatebox{-12}{$\mathcal{O}(N^{-2})$}}}
 \put(50,-1.5) {$N$}
 \put(-4,27) {\rotatebox{90}{$\|\underline{\tilde{c}} - \underline{c}\|_\infty$}}
 \end{overpic}
\caption{The maximum absolute error $\|\underline{\tilde{c}}-\underline{c}\|$, where $\underline{c}$ and $\underline{\tilde{c}}$
satisfy~\eqref{eq:inverseTest}. Here, the DHT was computed with direct summation 
and our algorithm with $\epsilon=10^{-3},10^{-8},10^{-15}$. As $N\rightarrow\infty$ 
the DHT is its own inverse, up to a diagonal scaling. For $N<700$ the error 
$\|\underline{\tilde{c}}-\underline{c}\|$ decays like $\mathcal{O}(N^{-2})$ because~\eqref{eq:selfinverse}
only holds asymptotically as $N\rightarrow\infty$. For $N> 700$ the error 
$\|\underline{\tilde{c}}-\underline{c}\|$ grows like $\mathcal{O}(N^{3/2})$ from 
numerical error accumulation.}
 \label{fig:HankelResults2}
\end{figure}

\section*{Conclusion}
An asymptotic expansion of Bessel functions for large arguments is 
carefully employed to derive a numerically stable $\mathcal{O}(N(\log N)^2/\log\!\log N)$ 
algorithm for evaluating Schl\"{o}milch and Fourier--Bessel expansions as well as 
computing the discrete Hankel transform. All algorithmic parameters are selected based 
on error bounds to achieve a near-optimal computational cost for any accuracy goal. 
For a working accuracy of $10^{-15}$, the algorithm is faster than direct summation for 
evaluating Schl\"{o}milch expansions when $N\geq 100$, Fourier--Bessel 
expansions when $N\geq 700$, and the discrete Hankel transform when $N\geq 6,\!000$. 

\section*{Acknowledgments}
I thank the authors and editors of the Digital Library of 
Mathematical Functions (DLMF). 
Without it this paper would have taken longer. I also thank 
the referees for their time and consideration.

\end{document}